%% file: weak-integrability.tex
\documentclass[a4paper]{amsart}
\input{definitions.tex}
\includecomments{true}

\begin{document}

\title[Weak Liouville-\Arnold{!} Theorems]{Weak Liouville-\Arnold{!} Theorems \& Their Implications}
\author{Leo T. Butler \& Alfonso Sorrentino}
\address{Department of Mathematics, Central Michigan University, Mt. Pleasant, MI, USA, 48859}
\email{l.butler@cmich.edu}
\address{Dipartimento di Matematica,
 Università degli Studi ``Roma Tre'',
 Largo S. Leonardo Murialdo 1,
 00146 Rome, Italy}
\email{sorrentino@mat.uniroma3.it}
\date{\today}
\thanks{L. B. thanks K. F. Siburg for helpful discussions and the MFO for its hospitality. A. S. would like to acknowledge the support of the {\it Herchel-Smith foundation} and the French ANR project {\it Hamilton-Jacobi et th\'eorie KAM faible}.}
\begin{abstract}
  This paper studies the existence of invariant smooth Lagrangian graphs for Tonelli Hamiltonian systems with symmetries. In particular, we consider Tonelli Hamiltonians with $n$ independent but not necessarily involutive constants of motion and obtain two theorems reminiscent of the Liouville-\Arnold{} theorem.
Moreover, we also obtain results on the structure of the configuration spaces of such systems that are reminiscent of results on the configuration space of completely integrable Tonelli Hamiltonians.
\end{abstract}
\maketitle

\section{Introduction}
\label{sec:introduction}

\begin{mcomment}[seen]
  I added an \verb#mcomment# environment for marginal comments. You
  can turn comments off, without deleting them, by changing
  \verb#\includecomments{true}# to \verb#\includecomments{false}# on
  line 3. You can stop the display of a comment by replacing \verb#\begin{mcomment}#
    with \verb#\begin{mcomment}[seen]#.

  There is no need to highlight your changes in red, version control
  (\verb#subversion#) lets me see them.

  You can delete this comment or mark it as seen.
\end{mcomment}

In the study of Hamiltonian systems, a special role is played by \textem{invariant Lagrangian manifolds}. 
These objects arise quite naturally in many physical and geometric problems and share a deep relation with the dynamics of the system and with the Hamiltonian itself.
Our concern in this paper is with Hamiltonian systems that possess invariant Lagrangian \textem{graphs}, or more precisely, with conditions that imply the existence of such graphs.
Specifically, we address the following question:
\begin{question}
\label{mainquestion}
When does a Hamiltonian system possess an invariant smooth Lagrangian graph?
\end{question}

It is natural to expect that ``sufficiently'' symmetric systems ought to possess an abundance of invariant Lagrangian graphs.
Inspired by the results in \cite{Sor09}, this paper demonstrates, with two different notions of symmetry, conditions that imply the existence of such graphs.
This approach to Question \ref{mainquestion} leads us to two theorems which in important aspects mirror the classical theorem
of Liouville-\Arnold{}. 
While on the one hand the first of these theorems can be seen as a (non-trivial) generalisation of the main theorem in \cite{Sor09} (see Remark \ref{remmaintheorem} for more details), on the other hand our present analysis extends well beyond, providing a much deeper insight into the nature and the properties of the so-called {\it weakly-integrable systems}. There is a large literature on the structure of the configuration space of a completely integrable
Tonelli Hamiltonian, see \interalia \cite{MR556099,MR897007,MR1428676,MR2136534}; in pursuing the analogy between the present paper's \textem{weak} Liouville-\Arnold{} theorems
and the classical theorem, we have proven two results on the topological structure of the configuration spaces of weakly-integrable
systems. Indeed, we believe that the following is an interesting question
\begin{question}
  \label{mainquestion-config}
  If a Hamiltonian system possesses an invariant smooth Lagrangian graph, what is true of its configuration space?
\end{question}
\noindent
To address each of these questions, we use two notions of ``symmetry'' in this paper. Let us introduce those:

\subsection*{Classical Symmetries}
Let us recall some terminology used to describe classical symmetries.
The cotangent bundle, $T^*M$, of a smooth manifold $M$ is equipped with a canonical Poisson structure $\pb{\cdot}{\cdot}$.
Given a smooth function $H$, the vector field $X_H = \pb{H}{\,}$ is a \textem{Hamiltonian system} with Hamiltonian $H$.
The skew-symmetry of $\pb{\cdot}{\cdot}$ implies that if $\pb{H}{F} \equiv 0$, then the vector field $X_H$ is tangent to the level sets of $F$; and, the Jacobi identity implies it commutes with $X_F$.
In such a situation, these Hamiltonians are said to \textem{Poisson-commute}, or be \textem{in involution}, and $F$ is said to be a \textem{constant of motion}, or \textem{first integral}.
The Liouville-\Arnold{} theorem describes the situation when $H$ has $n$ independent, Poisson commuting integrals.

\begin{theorem*}[Liouville-\Arnold{}] Let $(V,\omega)$ be a symplectic manifold with $\dim V=2n$ and let $H: V \longrightarrow \R$ be a proper Hamiltonian. Suppose that there exists $n$ integrals of motion $F_1, \ldots, F_n : V \longrightarrow   \R$ such that:
\begin{itemise}
\item $F_1,\;\ldots,\; F_n$ are $C^2$ and functionally independent almost everywhere on $V$;
\item $F_1,\ldots, F_n$ are pairwise in involution, \ie $\pb{F_i}{F_j}=0$ for all $i,j=1,\ldots n$.
\end{itemise}
Suppose the non-empty regular level set $\Lambda_{a} := \set{F_1=a_1,
  \ldots, F_n=a_n}$ is connected. Then $\Lambda_a$ is an $n$-torus, $\T^n$
and there is a neighbourhood $\cO$ of $0 \in H^1(\Lambda_a;\R)$ such
that for each $c' \in \cO$ there is a unique smooth Lagrangian
$\Lambda_{c'}$ that is a graph over $\Lambda_a$ with cohomology class
$c'$. Moreover, the flow of $X_H|\Lambda_{c'}$ is a rigid rotation.
\end{theorem*}

\begin{remark*}
  There are numerous proofs of this theorem in its modern formulation, see
  \interalia \cite{Mineur,Arnoldbook,MR1185565,MR596430,MR0451301}.
  The map $F:=(F_1,\ldots, F_n)$ is referred to as an \textem{integral
    map}, \textem{first-integral map} or a \textem{momentum map}.
The invariance of the level set $\Lambda_a$  simply follows from $F$ being an integral of motion; the fact that it is a Lagrangian torus and that the Hamiltonian flow is conjugate to a rigid rotation, strongly relies on these integrals being pairwise in involution and independent.
\end{remark*}

Inspired by the Liouville-\Arnold{} theorem, we address Question \ref{mainquestion} in the case of systems that possess a sufficiently large number of symmetries. Let us recall the definition of a \textem{weakly integrable} system.

\begin{definition}[Weak integrability \cite{Sor09}]
  \label{def:1}
  Let $H \in C^2(T^*M)$. If
  there is a $C^2$ map $F : T^*M^n \longrightarrow \R^n$ whose singular set is nowhere
  dense, and $F$ Poisson-commutes with $H$, then we say that $H$ is
  \textem{weakly integrable}.
\end{definition}

\subsection{Results}
\label{sec:results}

Recall that a Hamiltonian $H\in C^2(T^*M)$ is \textem{Tonelli} if it is fibrewise strictly convex and enjoys fibrewise superlinear growth\footnote{Section \ref{sec:am-theory} provides a synopsis of Mather theory and Fathi's weak KAM theory.}.
This paper's first result is

\begin{theorem}[Weak Liouville-\Arnold{}]
\label{mainthm}
Let $M$ be a closed manifold of dimension $n$ and $H:T^*M \longrightarrow \R$ a weakly integrable Tonelli Hamiltonian with integral map
$F: T^*M \longrightarrow \R^n$. If for some cohomology class $c\in H^1(M;\R)$ the corresponding Aubry set $\cA^*_c \subset \reg{F}$, then there exists an open neighborhood $\cO$ of $c$ in $H^1(M;\R)$ such that the following holds.
\begin{itemise}
\item For each $c'\in \cO$ there exists a smooth invariant Lagrangian graph $\Lambda_{c'}$ of cohomology class $c'$, which admits the structure of a smooth $\T^d$-bundle over a base $B^{n-d}$ that is parallelisable, for some $d>0$. 
\item The motion on each $\Lambda_{c'}$ is Schwartzman strictly ergodic (see \cite{FGS}), \ie all invariant probability measures have the same rotation vector and the union of their supports equals $\Lambda_{c'}$. In particular, all orbits  are conjugate by a smooth diffeomorphism isotopic to the identity.
\item Mather's $\alpha$-function $\alpha_H: H^1(M;\R) \longrightarrow \R $ is differentiable at all  $c'\in\cO$ and its convex conjugate
$\beta_H: H_1(M;\R) \longrightarrow \R$ is differentiable at all  rotation vectors $h\in\partial \alpha_H(\cO)$, where $\partial \alpha_H(\cO)$ denotes the set of subderivatives of $\alpha_H$ at some element of $\cO$.
\end{itemise}
\end{theorem}

\begin{remark*}
  \begin{itemise}[0mm]
  \tem We named this theorem as such because it drops the involutivity hypothesis of the classical theorem and still obtains results that are quite analogous.
  \tem The theorem remains true if one replaces the hypothesis $\cA^*_c \subset \reg{F}$ with $\cM^*_c \subset \reg{F}$, where $\cM^*_c$ denotes the {\it Mather set}. 
  \tem We conjecture that {weak integrability  implies that $\dim H^1(M;\R) \leq \dim M$ with equality if and only if $M$ is a torus} even without the a priori assumption $\cA^*_c \subset \reg{F}$. 
    \end{itemise}
\end{remark*}

\begin{remark}\label{remmaintheorem}
This theorem extends and improves  the main result in \cite{Sor09} in many non-trivial respects. 
 \begin{itemise}
	\item First of all, we provide a description of the topological structure of these invariant Lagrangian graphs $\Lambda_{c'}$, showing that they admit the structure of a smooth $\T^d$-bundle over a parallelisable base.
	\item Then, we prove that each $\Lambda_{c'}$ has a well-defined rotation vector, which implies the differentiability of Mather's $\alpha$ function.
	\item Moreover, we prove that the flow of $X_H|\Lambda_{c'}$ is a
  	rotation on the $\T^d$ fibres of $\Lambda_{c'}$ with rotation vector
  	$h_{c'}=\partial \alpha_H(c')$, where $\partial \alpha_H(c')$ is the
  	derivative of $\alpha_H$ at $c'$. This is analogous to what happens in the classical Liouville-\Arnold{}
  	theorem, where the rotation vector is the derivative of $H$ at $c'$.
  \end{itemise}
\end{remark}

Let us then pursue the analogy with the Liouville-\Arnold{} theorem and
complete integrability and turn now to the implications of Theorem
\ref{mainthm} for the topology of the configuration space $M$. 
Recall that a smooth manifold is \textem{irreducible} if,
when written as a connect sum, one of the summands is a standard
sphere.
In $3$-manifold topology, a central role is played by those closed
$3$-manifolds which contain a non-separating \textem{incompressible}
surface, or dually, which have non-vanishing first Betti number. Such
manifolds are called \textem{Haken}; it is an outstanding conjecture
that every irreducible $3$-manifold with infinite fundamental group
has a finite covering that is Haken \cite[Questions
1.1--1.3]{MR739142}. This conjecture is implied by the 
virtually fibred conjecture \cite{MR2399130}. Given the proof of the
geometrisation conjecture, the virtual Haken conjecture is proven for
all cases but hyperbolic $3$-manifolds. Thurston and Dunfield have
shown there is good reason to believe the conjecture is true in this
case \cite{MR1988291}.

\begin{theorem}
  \label{maincor}
  Assume the hypotheses of Theorem \ref{mainthm}. Then $M$ is
  diffeomorphic to a trivial $\T^d$-bundle over a parallelisable base
  $B$ such that all finite covering spaces of $B$ have zero first
  Betti number. Therefore
  \begin{itemise}
    \item $\dim M \leq 3$ implies that $M$ is diffeomorphic to a torus;
    \item $\dim M=4$ implies, assuming the virtual Haken conjecture,
      that $M$ is diffeomorphic to either $\T^4$ or $\T^1 \times E$,
      where $E$ is an orientable $3$-manifold finitely covered by
      $S^3$.
    \item If $\dim H^1(M;\R) \geq \dim M$, then $\dim H^1(M;\R) = \dim M$
      and $M$ is diffeomorphic to $\T^n=\R^n/\Z^n$.
  \end{itemise}
\end{theorem}

\subsection*{Non-classical Symmetries}

In the second part of this article, we investigate the case in which the system's
symmetries are not classical and do not come from conserved
quantities, but are induced by invariance under the action of an
amenable Lie group on the universal cover of the manifold. This action
need not descend to the quotient and is generally only evident in
statistical properties of orbits. In particular, these symmetries may
only manifest themselves in the structure of the action-minimizing
sets.

Recall that a topological group is
\textem{amenable} if it admits a left-invariant, finitely additive,
Borel probability measure. Due to the Levi decomposition, an amenable
Lie group is a semi-direct product of its solvable radical and a
compact subgroup. A solvable Lie group is said to be
\textem{exponential} or \textem{\typee{}} if the exponential map of
the Lie algebra is surjective; we will say an amenable Lie group is of
\typee{} if its radical is of \typee. For each bi-invariant $1$-form
$\phi$ on the simply-connected amenable Lie group $G$ with lattice subgroup $\Gamma$,
let $\Lambda_c = \Gamma\cdot\Graph{\phi} \subset T^*(\Gamma \backslash G)$
be the Lagrangian graph of cohomology class $c$. The union of such
graphs is a submanifold $\cM \subset T^*(\Gamma \backslash G)$ naturally
diffeomorphic to $H^1(\Gamma \backslash G;\R) \times \Gamma \backslash
G$ and this diffeomorphism sends $\Lambda_c$ to $\set{c} \times \Gamma
\backslash G$.

\begin{theorem}
  \label{thm:2}
  Let $G$ be a simply-connected amenable Lie group and let $\Gamma \lhd G$ be
  a lattice subgroup, $M=\Gamma\backslash G$ and $H$ be induced by a
  left-invariant $C^r$ Tonelli Hamiltonian on $T^*G$. Then
  \begin{itemise}
    \item for all $c \in H^1(M;\R)$, the Mather set $\cM^*_c(H)$ equals
    the Lagrangian graph $\Lambda_c$;
    \item the flow of $X_H | \Lambda_c$ is a right-translation by a
    $1$-parameter subgroup of $G$;
    \item the motion on $\Lambda_c$ is Schwartzman strictly ergodic;
    \item Mather's $\alpha$ function $\alpha_H : H^1(M;\R) \to \R$ is $C^r$.
  \end{itemise}
\end{theorem}

\begin{remark}
  \label{re:2}
  (i--ii) provide analogues to the Lagrangian tori and action-angle
  coordinates in the classical Liouville-\Arnold{} theorem. However,
  there are some oddities: for example, it is possible that these
  right-translations have positive topological entropy. Indeed, this
  is \textem{exactly} what happens in the $Sol$ $3$-manifold examples
  of Bolsinov-Taimanov \cite{MR1741286} (in that example, all such Tonelli
  Hamiltonians are also completely integrable).
  Moreover, in this case we can prove that the {\it frequency map} $\partial \alpha_H$ has the same regularity as the Hamiltonian vector field (see iv). Whether or not the same property holds in Theorem \ref{mainthm} remains an open question. Observe that this problem is strictly related to the regularity of the family $\{\Lambda_{c'}\}$ as a function of $c'$.
\end{remark}

To prove this theorem we introduce a generalised notion of rotation vector and  a novel averaging procedure (see section \ref{sec:amenable-group}), which are likely to be of independent interest.\\

Finally, under some additional assumptions we can complete Theorem \ref{thm:2} and prove the following implications for the topology of the configuration space.

\begin{theorem}
  \label{thm:2-implications}
  Assume the hypotheses of Theorem \ref{thm:2}. Assume additionally
  that $G$ is of \typee{}.
  If $H$ is weakly integrable with integral map $F : T^* M \to \R^n$ and there is a $C^1$ Lagrangian graph
  $\Lambda \subset H^{-1}(h)$ and $\Lambda \cap \reg{F} \neq
  \emptyset$, then $M$ is finitely covered by a compact reductive
  Lie group with a non-trivial centre.
\end{theorem}

\subsection{Methodological remarks}
\label{sec:meth-rmks}

From a superficial perspective, theorems \ref{mainthm} and \ref{thm:2} appear quite distinct.
However, they are quite intimately related.
In trying to weaken the Liouville-\Arnold{} theorem, one must find a substitute for its involutivity hypothesis.
Our substitute is to apply Mather theory and Fathi's weak KAM theory to
systems with symmetry.
It is natural to wonder if non-classical symmetries might also leave traces of their existence in the form of invariant Lagrangian graphs--complete solutions of the Hamilton-Jacobi equation.
Theorem~\ref{thm:2} shows that certain types of symmetry, that need not be associated with conserved quantities, do manifest themselves in this fashion.

\section{Action-minimizing sets and integrals of motion}
\label{sec:am-theory}

In the study of weakly integrable systems, or more generally of convex and superlinear Hamiltonian systems, the main idea behind dropping the hypothesis on the involution of the integrals of motion 
consists in studying the relationship between the existence of integrals of motion and the structure of some invariant sets obtained by {\em action-minimizing methods}, which are generally called {\em Mather, Aubry} and {\em Ma\~n\'e sets}.

In this section we want to provide a brief description of this theory, originally developed by John Mather, and the main properties of these sets. We refer the reader to  \cite{Fathibook,Mather91, Mather93, ManeI, SorNotes} for more exhaustive presentations of this material. 
Roughly speaking these action-minimizing sets represent a generalisation of invariant Lagrangian graphs, in the sense that, although they are not necessarily submanifolds, nor even connected,  they still enjoy many similar properties. 
What  is crucial for our study of weakly integrable systems is that these sets have an intrinsic Lagrangian structure, which implies many of their symplectic properties, including a forced {\em local involution} of the integrals of motion, as noticed in \cite{Sor09}.

More specifically, we are interested in studying the existence of {\em action-minimizing invariant probability measures}  and {\em action-minimizing orbits} in the following setting. 

Let $H: T^* M \rightarrow \R$ be a $C^2$ Hamiltonian, which is strictly convex and uniformly superlinear in the fibres. $H$ is called  a {\em Tonelli Hamiltonian}. This Hamiltonian defines a vector field on $T^*M$, known as {\em Hamiltonian vector field}, that can be defined as the unique vector field $X_H$ such that $\omega(X_H,\cdot)= dH$, where $\omega$ is the canonical symplectic form on $T^*M$. We  call the associated flow {\em Hamiltonian flow} and denote it by $\Phi^t_H$.

To any Tonelli Hamiltonian system one can also associate an equivalent dynamical system in the tangent bundle ${TM}$, called {\em Lagrangian system}. Let us consider the associated {\em Tonelli Lagrangian}  $L: T M \rightarrow \R$, defined as
$L(x,v):= \max_{p\in T_x^*M} \left( \langle p,v \rangle - H(x,p) \right)$. It is possible to check that $L$ is also 
 strictly convex and uniformly superlinear in the fibres. In particular this Lagrangian defines a flow on ${T M}$, known as {\em Euler-Lagrange flow} and denoted by $\Phi_L^t$, which can be obtained by integrating the so-called {\em Euler-Lagrange equations}: 
 $$
 \frac{d}{dt} \frac{\partial L}{\partial v}(x,v) = \frac{\partial L}{\partial x}(x,v).
 $$ 
The Hamiltonian and Lagrangian flows are totally equivalent from a dynamical system point of view, in the sense that there exists a {conjugation} between the two. In other words, there exists a diffeomorphism $\cL_L: T M \longrightarrow T^*M$, called {\em Legendre transform}, defined by
$\cL_L(x,v)=(x, \frac{\partial L}{\partial v}(x,v))$, such that $\Phi_H^t = \cL \circ \Phi_L^t \circ \cL^{-1}$.

In classical mechanics, a special role in the study of Hamiltonian dynamics is represented by invariant Lagrangian graphs, \ie graphs of the form
$\Lambda:=\{(x,\eta(x)):\; x\in M\}$ that are Lagrangian (\ie $\omega\big|_{\Lambda}\equiv 0$) and invariant under the Hamiltonian flow $\Phi_H^t$. Recall that being a Lagrangian graph in $T^*M$ is equivalent to say that $\eta$ is a closed $1$-form (\cite[Section 3.2]{Cannas}). These graphs satisfy many interesting properties, but unfortunately they are quite rare. The theory that we are going to describe aims to provide a generalisation of these graphs; namely, we shall construct several compact invariant subsets of the phase space, which are not necessarily submanifolds, but that are contained in Lipschitz Lagrangian graphs and enjoy similar interesting properties.

Let us start by recalling that the Euler-Lagrange flow $\Phi_L^t$ can be also characterised in a more variational way, introducing the so-called {\em Lagrangian action}. Given an absolutely continuous curve $\gamma:[a,b]\longrightarrow M$, we define its action as $A_L(\gamma) =\int_a^b L(\gamma(t),\dot{\gamma}(t))\,dt$. It is a classical result that 
a curve $\gamma:[a,b]\longrightarrow M$ is a solution of the Euler-Lagrange equations if and only if it is a critical point  of $A_L$, restricted to the set of all curves connecting $\gamma(a)$ to $\gamma(b)$ in time $b-a$. However, in general, these extrema are not minima (except if their time-length $b-a$ is very small). Whence the idea of considering minimizing objects and seeing if - whenever they exist - they enjoy special properties or possess a more distinguished structure.

Mather's approach is indeed based  this idea and is concerned with the study of invariant probability measures and orbits that minimize the Lagrangian action (by action of a measure, we mean the collective average action of the orbits in its support, \ie the integral of the Lagrangian against the measure).
It is quite easy to prove (see \cite[Lemma 3.1]{FGS} and \cite[Section 3]{SorNotes}) that invariant probability measures (resp. Hamiltonian orbits) contained in an invariant Lagrangian graph $\Lambda$ (actually its pull-back using $\cL$) minimize the {Lagrangian action} of $L-\eta$, which we shall denote $A_{L-\eta}$, over the set 
$\calM(L)$ of all invariant probability measures for $\Phi_L^t$ (resp. over the set of all curves with the same end-points and defined for the same time interval).
This idea  of changing Lagrangian (which is at the same time a necessity) plays an important role as it allows one to magnify some motions rather than others. For instance, consider the case of an integrable system: one cannot expect to recover all these motions (which foliate the whole phase space) by just minimizing the same Lagrangian action! What is important to point out is that even if we modify $L$, because of the closedness of $\eta$  we do not change the associated Euler-Lagrange flow, \ie $L-\eta$ has the same Euler-Lagrange flow as $L$ (see \cite[p. 177]{Mather91} or \cite[Lemma 4.6]{SorNotes}). This is a crucial step in Mather's approach 
in  \cite{Mather91}:  consider a family of modified Tonelli Lagrangians given by  $L_{\eta}(x,v)=L(x,v)-\langle \eta(x),v\rangle$, where $\eta$ is a closed $1$-form on $M$. These Lagrangians  have the same Euler-Lagrange flow as $L$, but different action-minimizing orbits and measures. Moreover, these action-minimizing objects depend only on the cohomology class of $\eta$ \cite[Lemma p.176]{Mather91}.

Hence,  for each $c \in{\rm H}^1(M;\R)$, if we choose $\eta_c$ to be any smooth closed $1$-form on $M$ with cohomology class $[\eta_c]=c$, we can study action-minimizing invariant probability measures (or orbits) for $L_{\eta_c}:= L-\eta_c$. In particular, this allows one to define  several compact invariant subsets of $TM$:

\begin{itemize}
\item $\widetilde{\cM}_c(L)$, the {\em Mather set of cohomology class $c$}, given by the union of the supports of all invariant probability measures that minimize the action of $L_{\eta_c}$ ($c$-{\em action minimizing measure} or {\em Mather's measures of cohomology class} $c$). See \cite{Mather91}.
\item $\widetilde{\cN}_c(L)$, the {\em Ma\~n\'e set of cohomology class $c$}, given by the union of  all orbits that minimize the action   of $L_{\eta_c}$ on the finite time interval $[a,b]$, for any $a<b$. These orbits are called {\em $c$- global minimizers} or {\em $c$-semi static curves}. \cite{Mather91, Mather93, ManeI}.
\item $\widetilde{\cA}_c(L)$, the {\em Aubry set of cohomology class $c$}, given by the union of  the so called $c-${\em regular  minimizers}  of $L_{\eta_c}$ (or $c$-{\em static curves}). These are special kind of $c$-global minimizers that, roughly speaking, do not only minimize the Lagrangian action to go from the starting point to the end-point, but that - up to a change of sign - also minimize  the action to go backwards, \ie from the end-point to the starting one. A precise definition would require a longer discussion. Since we are not using this definition in the following, we refer the interested reader to \cite{Mather93, ManeI, SorNotes}.
\end{itemize}

\begin{remark}
\begin{itemise}
\item These sets are non-empty, compact, invariant and moreover they satisfy the following inclusions:
$$\widetilde{\cM}_c(L) \subseteq \widetilde{\cA}_c(L) \subseteq \widetilde{\cN}_c(L) \subseteq TM\,.$$
\item The most important feature of the Mather set and the Aubry set is the so-called {\em graph property}, namely they are contained in Lipschitz graphs over $M$ ({\em Mather's graph theorem} \cite[Theorem 2]{Mather91}). More specifically, if $\pi: TM \rightarrow M$ denotes the canonical projection along the fibres, then $\pi|{\widetilde{\cA}_c(L)}$ is injective and its inverse $\big(\pi|\widetilde{\cA}_c(L)\big)^{-1}\!\!\!: \pi\big(\widetilde{\cA}_c(L)\big) \longrightarrow \widetilde{\cA}_c(L)$ is Lipschitz. The same is true for the Mather set (it follows from the above inclusion). 
Observe that in general the Ma\~n\'e set does not  necessarily satisfy the graph property.
\item As we have mentioned above, when there is an invariant Lagrangian graph $\Lambda$ of cohomology class $c$ (\ie it is the graph of a closed $1$-form of cohomology class $c$), then $\widetilde{\cN}_c(L)=\cL^{-1}(\Lambda)$. A priori $\widetilde{\cA}_c(L) \subseteq \cL_L^{-1}(\Lambda)$ and $\widetilde{\cM}_c(L)\subseteq \cL_L^{-1}(\Lambda)$. In particular  $\widetilde{\cM}_c(L)= \cL_L^{-1}(\Lambda)$ if and only if the whole Lagrangian graph is the support of an invariant probability measure (\ie the motion on it is recurrent).
\item Similarly to what happens for invariant Lagrangian graphs, the energy $E(x,v)=\left\langle \frac{\partial L}{\partial v}(x,v), v \right\rangle - L(x,v)$ (\ie the pull-back of the Hamiltonian to $TM$ using the Legendre transform) is constant on these sets, \ie for any $c\in H^1(M;\R)$ the corresponding sets  lie in the same energy level $\a_H(c)$. Moreover, Carneiro \cite{Carneiro} proved a characterization of this energy value in terms of the minimal Lagrangian action of $L-\eta_c$. More specifically:
$$
\a_H(c) = - \min_{\mu \in \calM(L)} A_{L-\eta_c}(\mu).
$$
This defines a function $\a_H: H^1(M;\R)\longrightarrow \R$ that is generally called {\em Mather's $\a$-function} or {\em effective Hamiltonian} (see also \cite[p. 177]{Mather91}). 
\item  It is possible to show that Mather's $\a$-function is convex and superlinear \cite[Theorem 1]{Mather91}. In particular, one can consider its convex conjugate, using  Fenchel duality, which is a function on the dual space $(H^1(M;\R))^*\simeq H_1(M;\R)$ and is given by:
\beqano
\beta_H: H_1(M;\R) &\longrightarrow & \R \\
h &\longmapsto& \max_{c\in H^1(M;\R)} \left(\langle c, h\rangle - \alpha_H(c) \right).
\eeqano
This function is also convex and superlinear and is usually called {\em Mather's $\beta$-function}, or {\em effective Lagrangian}. It has also a meaning in terms of the minimal Lagrangian action. In fact, one can interpret elements in $H_1(M;\R)$ as rotation vectors of invariant probability measures \cite[p. 177]{Mather91} (or `Schwartzman asymptotic cycles'  \cite{Schwartzman}). In particular $\beta_H(h)$ represents the minimal Lagrangian action of $L$ over the set of all invariant probability measures with rotation vector $h$. Observe that in this case we do not need to modify the Lagrangian, since the constraint on the rotation vector will play somehow the role of the previous modification (it is in some sense the same idea as with Lagrange multipliers and constrained extrema of a function). We refer the reader to \cite{Mather91, SorNotes} for a more detailed discussion on the relation between these two different kinds of action-minimizing processes.
\end{itemise}
\end{remark}

Using the duality between Lagrangian and Hamiltonian, 
via the {Legendre transform} introduced above, one can define
the analogue of the Mather, Aubry and Ma\~n\'e sets in the cotangent bundle, simply considering 
$$\cM^*_c(H)=\cL_L\big(\widetilde{\cM}_c(L)\big),\qquad
\cA^*_c(H)=\cL_L\big(\widetilde{\cA}_c(L)\big) \qquad {\rm and} \qquad
\cN^*_c(H)=\cL_L\big(\widetilde{\cN}_c(L)\big).
$$
These sets  continue to satisfy the properties mentioned above, including the graph theorem.
Moreover, it follows from Carneiro's result \cite{Carneiro}, that they are contained in the energy level $\{H(x,p)=\a_H(c)\}$.
However, one could try to define these objects directly in the cotangent bundle. For any cohomology class $c$, let us fix a representative $\eta_c$. Observe  that if  $\Lambda:=\{(x,\eta(x)):\; x\in M\}$ is an invariant Lagrangian graph of cohomology class $c$, \ie $\eta=\eta_c +du$ for some $u:M\rightarrow \R$, then $H(x, \eta_c+du(x))= {\rm const}.$ Therefore, the Lagrangian graph is a solution (and of course a subsolution) of Hamilton-Jacobi equation $H(x, \eta_c+du(x))= k$, for some $k \in \R$. In general solutions of this equation, in the classical sense,  do not exist. However Albert Fathi proved  that it is always possible to find {\em weak solutions}, in the {\em viscosity sense}, and use them to recover the above results. This theory, that can be considered as the analytic counterpart of the variational approach discussed above, is nowadays called {\em weak KAM theory}. We refer the reader to \cite{Fathibook} for a more complete and precise presentation. 

It turns out that for a given cohomology class $c$ these weak solutions can exist only in a specific energy level, that - quite surprisingly - coincides with Mather's value $\a_H(c)$. This is also the least energy value for which Hamilton-Jacobi equation can have subsolutions:
\begin{equation}
  \label{eq:hj-inequation}
H(x,\eta_c + du(x)) \leq k
\end{equation}
where $u\in C^1(M)$. Observe that the existence of  $C^1$-subsolutions corresponding to $k=\a_H(c)$ is a non-trivial result due to Fathi and Siconolfi \cite{FathiSiconolfi}. Moreover they proved that these subsolutions are dense in the set of Lipschitz subsolutions. We shall call these subsolutions, {\em $\eta_c$-critical subsolutions}. Patrick Bernard \cite{Bernardc11} improved this result proving the existence and the denseness of $C^{1,1}$ $\eta_c$-critical subsolutions, which is the best result that one can generally expect to find. The main problem in fact is represented by the Aubry set itself, that plays the role of a {\em non-removable intersection} (see also \cite{PPS}). More specifically, for any  $\eta_c$-critical subsolution $u$, the value of  $\eta_c + d_xu$ is prescribed on $\pi(\cA_c^*(H))$, where $\pi:T^*M \longrightarrow M$ is the canonical projection. Therefore, if the Aubry set is not sufficiently smooth (it is at least Lipschitz), then these subsolutions cannot be smoother. However, on the other hand this obstacle provides a new characterization of the Aubry set in terms of these subsolutions. Namely, if one denotes  by $\cS_{\eta_c}$ the set of $C^{1,1}$ $\eta_c$-critical subsolutions, then: 
\beqa{Aubry} {\cA}_c^*(H)=\bigcap_{u\in \cS_{\eta_c}} \left\{(x,\eta_c + d_xu):\; x\in M\right\}.\eeqa

As we have already recalled,  in $T^*M$, with the standard symplectic form, there is a $1$-$1$ correspondence between Lagrangian graphs and closed $1$-forms (see for instance \cite[Section 3.2]{Cannas}). Therefore, we could interpret the graphs of the differentials of these critical subsolutions as {Lipschitz} Lagrangian graphs in $T^*M$. Therefore the Aubry set can be seen as the intersection of these {distinguished}  Lagrangian graphs and it is exactly this property that provides to this set the intrinsic Lagrangian structure mentioned above and that will play a crucial role in our proof. 

In \cite{Sor09}, in fact,  Sorrentino used this characterization to study the relation between the existence of integrals of motion and the size of the above action-minimizing sets.
Let $H$ be a Tonelli Hamiltonian on $T^*M$ and let $F$ be an integral of motion of $H$. 
If we denote by $\Phi_H$ and $\Phi_F$ the respective flows, then:

\begin{proposition}[see Lemma 2.2 in \cite{Sor09}] \label{prop:invariance}
The Mather set $\cM^*_c(H)$ and the Aubry set $\cA^*_c(H)$ are invariant under the action of $\Phi^t_F$, for each $t\in \R$ and for each $c\in H^1(M;\R)$. 
\end{proposition}

Moreover one can study the implications of the existence of {\em independent} integrals of motion, \ie integrals of motion whose differentials are linearly independent, as vectors, at each point of these sets.  It follows from the above proposition that this  relates to  the size of the Mather and Aubry sets of $H$. 
In order to make clear what we mean by the `size' of these sets, let us introduce some notion of tangent space.
We call {\em generalised tangent space} to $\cM^*_c(H)$ (resp.~$\cA^*_c(H)$) at a point $(x,p)$, the set of all vectors that are tangent to curves in $\cM^*_c(H)$ (resp.~$\cA^*_c(H)$) at $(x,p)$. We denote it by $T^G_{(x,p)}\cM^*_c(H)$ (resp.~$T^G_{(x,p)}\cA^*_c(H))$ and define its {\em rank} to be the largest number of linearly independent vectors that it contains. Then:

\begin{proposition}[See Proposition 2.4 in \cite{Sor09}]\label{Prop1Sor09}
Let $H$ be a Tonelli Hamiltonian on $T^*M$ and suppose that there exist $k$ independent integrals of motion on $\cM^*_c(H)$ (resp. $\cA^*_c(H)$).
Then, $\rank T^G_{(x,p)}\cM^*_c(H) \geq k$ (resp. $\rank T^G_{(x,p)}\cA^*_c(H) \geq k$)  at all points $(x,p)\in \cM^*_c(H)$ (resp. $(x,p)\in \cA^*_c(H)$).
\end{proposition}

\begin{remark}\label{RemSor09} In particular, the existence of the maximum possible number of integrals of motion (\ie $k=n$) implies that these sets are invariant smooth Lagrangian graphs (see \cite[Remark 3.5]{Sor09} or \cite[Lemma 3.4 and Lemma 3.6]{Sor09}). In particular, smoothness is a consequence of the fact that these graphs lie in level sets of the integral map, which is non-degenerate. \end{remark}

However the most important peculiarity of these action-minimizing sets observed in \cite{Sor09}, at least as far as we are concerned,  is that they force the integrals of motion to Poisson-commute on them. In fact, using the characterization of the Aubry set in terms of critical subsolutions of Hamilton-Jacobi  and its symplectic interpretation given above (see (\ref{Aubry}) and the subsequent comment), one can recover the involution property of the integrals of motion, at least locally.

\begin{proposition}[See Proposition 2.7 in \cite{Sor09}]\label{prop2Sor09}
Let $H$ be a Tonelli Hamiltonian on $T^*M$ and let $F_1$ and $F_2$ be two integrals of motion. Then for each $c\in H^1(M;\R)$ we have that $\{F_1,F_2\}(x,\hat{\pi}_c^{-1}\!(x))=0$ for all $x\in \overline{{\rm Int}\big(\cA_c(H)\big)}$, where $\hat{\pi}_c=\pi|\cA^*_c(H)$   and $\cA_c(H)=\pi\big(\cA^*_c(H)\big)$.
\end{proposition}

\begin{remark}
Observe that the above set $\overline{{\rm Int}\big(\cA_c(H)\big)}$ may be empty. What the proposition says is that whenever it is non-empty, the integrals of motion are forced to Poisson-commute on it. In the cases that we shall be considering hereafter, $\cA_c(H)=M$ and therefore it is not empty.
\end{remark}

\section{Proof of Theorem \ref{mainthm}}
\label{sec:proofs}

\begin{proposition}
  \label{thm:1} 
  Let $\Lambda \subset H^{-1}(h)$ be a $C^1$ Lagrangian graph. If $H$
  is a weakly integrable Tonelli Hamiltonian and $\Lambda \subset
  \reg{F}$, then $M$ admits the structure of a smooth $\T^d$-bundle
  over a parallelisable base $B^{n-d}$ for some $d>0$.
\end{proposition}

\begin{proof}[Proof (Proposition \ref{thm:1})]
  Since $\Lambda$ is a $C^1$ Lagrangian graph that lies in an energy
  surface of $H$, $\Lambda$ is the graph of a $C^1$ closed $1$-form
  $\lambda$ with cohomology class $c$. It follows that $\lambda$
  solves the Hamilton-Jacobi equation and from (\ref{Aubry}) that $\cA^*_c(H) \subseteq \Lambda$ (see also \cite[Section 3]{Sor09}). Moreover, Proposition \ref{Prop1Sor09} and Remark \ref{RemSor09} allow us to conclude that $\cA^*_c(H) = \Lambda$.
  Therefore, Proposition \ref{prop:invariance} implies
  that each vector field $X_{F_i}, i=1,\ldots,n$ is tangent to
  $\Lambda$. Let $Y=X_H|\Lambda$ and $Y_i=X_{F_i}|\Lambda$. Since
  $\Lambda \subset \reg{F}$, $\{Y_i\}$ is a framing of $T\Lambda$.
  
  Let $\phi^i$ (resp. $\phi$) be the flow of $Y_i$ (resp. $Y$). Let
  $\Gamma$ be the group of diffeomorphisms generated by the flows
  $\phi^i$ and $\phi$. The Stefan-Sussman orbit theorem implies that
  $\Lambda$ is the orbit of $\Gamma$: $\Lambda = \set{ \prod_{j=1}^m
    \phi^{i_j}_{t_j}(p)\ : \ t_j \in \R, m\in\N}$ for any
  $p\in\Lambda$~\cite{MR0310922,MR0321133,MR0362395}.
  Since $H$ Poisson-commutes with each of the $F_i$, the vector
  field $Y$ commutes with $Y_i$ for all $i$. Therefore, the flow
  $\phi$ of $Y$ commutes with each $\phi^i$, \ie $\phi$ lies in the
  centre $Z$ of $\Gamma$.

  Let $p \in \Lambda$ be a given point and $q \in \Lambda$ a second
  point. Let $\Phi = \prod_{j=1}^m \phi^{i_j}_{t_j} $ be an element in
  $\Gamma$ satisfying $\Phi(p)=q$. If $\varphi_t$ is a $1$-parameter
  subgroup of $Z$, then $\varphi_t(q) = \Phi(\varphi_t(p))$ for all
  $t\in\R$. Therefore, each orbit of $\varphi$ is conjugate by a
  smooth conjugacy isotopic to the identity. We have seen that $\phi_t
  \in Z$ for all $t$, and the above shows that each orbit of $\phi_t$
  (indeed, of $Z$) is conjugate.

  Define a smooth Riemannian metric $g$ on $\Lambda$ by defining
  $\{Y_i\}$ to be an orthonormal framing of $T\Lambda$. Then, we see
  that each element in $Z$ preserves $g$. Therefore $Z$ is a group of
  isometries of a compact Riemannian manifold. The closure of $Z$ in
  the group of $C^1$ diffeomorphisms of $\Lambda$, $\bar{Z}$, is
  therefore a compact connected abelian Lie group by the Montgomery-Zippin
  theorem~\cite{MR0073104}. Therefore, $\bar{Z}$ is a $d$-dimensional
  torus for some $d>0$ (since it contains the $1$-parameter group
  $\phi_t$).

  Since $Z$ centralises $\Gamma$, so does its closure
  $\bar{Z}$. Therefore, each orbit of $\bar{Z}$ is conjugate. It
  follows that $\bar{Z}$ acts freely on $\Lambda$. This gives
  $\Lambda$ the structure of a principal $\T^d$-bundle.

  Finally, let $p \in \Lambda$ be given. Possibly after a linear
  change of basis, we can suppose that $Y_i$, $i=1,\ldots,d$, is a
  basis of the tangent space to the $\T^d$-orbit through $p$, and
  $Y_i$, $i=d+1,\ldots,n$ is a basis of the orthogonal
  complement. Therefore, $Y_i$, $i=d+1,\ldots,n$ is a basis of the 
  orthogonal complement to the fibre at all points on
  $\Lambda$. Since each vector field $Y_i$ is $\T^d$-invariant, it
  descends to $B=\Lambda/\T^d$. Therefore, the vector fields on $B$
  induced by $Y_i$, $i=d+1,\ldots,n$, frame $TB$.
\end{proof}

\begin{remark}\label{rem:1}
  { A few remarks are in order. First, there is a $\xi \in
    \t=\liealg{\T^d}$ such that $\exp(t\xi)\cdot p = \phi_t(p)$ for
    all $t\in\R$ and $p\in\Lambda$. This follows from the fact that
    $\{\phi_t\} \subset \bar{Z}$ is a $1$-parameter
    subgroup. Therefore, there is a torus $T$ of dimension $c\leq d$
    which is the closure of $\{\exp(t\xi)\}$ in $\T^d$ such that each
    orbit closure of $\phi$ is the orbit of $T$. Second, for almost
    all constants $(\alpha_i) \in \R^d$, the vector field $Y_{\alpha}
    = Y + \sum_{i=1}^d \alpha_i Y_i$ will have dense orbits in each
    $\T^d$ orbit. Third,
    since each orbit of $\phi$ is conjugate by a diffeomorphism
    isotopic to $1$, the asymptotic homology of $\Lambda$ is unique (see \cite[Proposition A.1]{FGS}).
    Finally, if, as in Theorem \ref{mainthm}, one has an
    upper semicontinuous family of such Lagrangian graphs
    $\Lambda_{c'}$, then the dimension $d'$ of the torus is an
    upper semicontinuous function of $c'$.
  }
\end{remark}

\begin{proof}[Proof (Theorem \ref{mainthm})] 
Since $\cA^*_c$ is contained in the set of regular points of $F$,  it follows from Proposition \ref{Prop1Sor09} and Remark \ref{RemSor09} that the Aubry set $\cA_c^*$ is a $C^1$ invariant Lagrangian graph $\Lambda_c$ of cohomology class $c$ and that it coincides with the Mather set $\cM^*_c$ (see also \cite[Lemmas 3.4 \& 3.6]{Sor09}). Therefore, $\Lambda_c$ supports an invariant probability measure of full support.
In particular, since all $c$-critical subsolutions of the Hamilton-Jacobi equation \eqref{eq:hj-inequation}, with $k=\alpha_H(c)$, have the same differential on the (projected) Aubry set \cite[Theorem 4.11.5]{Fathibook}, it follows that, up to constants, there exists a unique $c$-critical subsolution, which is indeed a solution. It follows then that the Ma\~n\'e set $\cN_c^*=\cA^*_c$ (see \cite[Definition 5.2.5]{Fathibook}).  We can use the upper semicontinuity of the Ma\~n\'e set (see for instance \cite[Proposition 13]{Arn08}) to deduce that the Ma\~n\'e set corresponding to nearby cohomology classes must also lie in $\reg{F}$ (note in fact that in general the Aubry set is not upper semicontinuous \cite{Bern08}). Hence, there  exists an open neighborhood $\cO$ of $c$ in $H^1(M;\R)$ such that  $\cA_{c'}^* \subseteq \cN_{c'}^* \subset \reg{F}$ for all $c'\in \cO$ and applying the same argument as above, we can conclude that each $\cA_{c'}^*$ is a smooth invariant Lagrangian graph of cohomology class $c'$ and that it coincides with the Mather set $\cM^*_{c'}$.

At this point (i) and (ii) follow from Proposition \ref{thm:1} and Remark \ref{rem:1}.

The proof of (iii) is the same as in \cite[Corollary 3.8]{Sor09}, but in this case we also know that these graphs are Schwartzman uniquely ergodic, \ie all invariant probability measures on  $\Lambda_{c'}$ have  the same rotation vector $h_{c'} \in H_1(M;\R)$ (see Remark \ref{rem:1}). 
The differentiability of $\alpha_H$ follows then from  \cite[Corollary 3.6]{FGS}.
The differentiability of $\beta_H$ follows the disjointness of these graphs (see for instance \cite[Theorem 3.3]{FGS} or \cite[Remark 4.26 (ii)]{SorNotes}).

\end{proof}

\begin{proof}[Proof (Theorem \ref{maincor})]
Let $d$ be the largest dimension of the torus fibre of $\Lambda_{c}$
for $c \in \cO$. The upper semicontinuity of this dimension implies
that there is an open set on which the dimension of the fibre equals
$d$; without loss of generality, it can be supposed that this open set
is $\cO$. By (iii) of Theorem \ref{mainthm}, Mather's
$\alpha$-function is differentiable on $\cO$. Since $\alpha_H$ is a
locally Lipschitz function, it is continuously
differentiable on $\cO$. Therefore, the map
\begin{align*}
  \xymatrix{
    c \ar@{|->}[r]   & h=\partial \alpha_H(c), &&
    \cO \ar[r]^(.35){\partial \alpha_H} & H_1(M;\R)
  }
\end{align*}
is continuous and one-to-one (by \cite[Theorem 3.3]{FGS}) and
hence a homeomorphism onto its image.

Let $b_1(M) = \dim H_1(M;\R)$ be the first Betti number of $M$. Since
the rotation vector of $Y=X_H | \Lambda_c$ is the image of a cycle in
$H_1(\T^d;\R)$, $d \geq b_1(M)$. To prove that $d = b_1(M)$, we need a
few lemmata. (We follow the notation of the proof of Theorem
\ref{mainthm}.)

\begin{lemma}
  \label{le:h1b}
  $H_1(B;\R) = \set{0}$.
\end{lemma}
\begin{proof}
  Let $c \in \cO$. The rotation vector of $Y$ projects to $0$ in
  $H_1(B;\R)$, since $Y$ is tangent to the $\T^d$ fibres of
  $\Lambda_c$. But the projection map $\pi_c : \Lambda_c \to B$ is
  surjective on $H_1$ and $\pi_{c,*}(\partial \alpha_H(\cO))$ is open
  since $\pi_{c,*}$ is an open map. These facts imply $H_1(B;\R)$ is
  trivial.
\end{proof}

\begin{lemma}
  \label{le:unique-rotn-vec}
  For all $f \in F^* C^{\infty}(\R^n)$ and $c \in \cO$ the rotation
  set of $Y_f = X_f | \Lambda_c$ contains a unique point.
\end{lemma}
\begin{proof}
  The proof of Theorem \ref{mainthm} shows that there is a
  $\T^d$-connection on $\Lambda_c$ that permits one to decompose $Y_f$
  into a vertical component $Y_v$ and a horizontal component
  $Y_h$. Because $\Lambda_c$ is a level set of $F$, $Y_v$
  (resp. $Y_h$) is a linear combination of the basis
  $Y_i$, $i=1,\ldots,d$ (resp. $Y_i$,$i=d+1,\ldots,n$) with
  \textem{constant} coefficients. Therefore, $Y_v$ and $Y_h$ commute
  and the flow of $Y_f$ is a product of their commuting flows:
  $\phi^f_t = \phi^v_t \circ \phi^h_t$ (and $\phi^v_t$ lies in the
  centre of $\Gamma$). Since $H_1(B;\R)=\set{0}$, the rotation set of
  $Y_h$ is trivial, so the rotation set of $Y_f$ equals that of $Y_v$,
  which is a singleton from the proof of Proposition \ref{thm:1}.
\end{proof}

Suppose now that $d > b_1(M)$. Let $G = H + F^* \alpha \in F^*
C^{\infty}(\R^n)$ be a Tonelli Hamiltonian such that for a residual
set of $c \in \cO$, the vertical component of $Y_G = X_G | \Lambda_c$
generates a dense $1$-parameter subgroup of the torus fibre. It is
straightforward to see that such $G$ exist. Lemma
\ref{le:unique-rotn-vec} implies that $\alpha_G | \cO$ is
differentiable and therefore $\partial \alpha_G | \cO$ is a
homeomorphism onto its image. If $d > b_1(M)$, then there are distinct
$c,c' \in \cO$ such that the rotation vectors $\rho(\Lambda_c) =
\rho(\Lambda_{c'})$, which contradicts the injectivity of $\partial
\alpha_G$. Therefore, $d = b_1(M)$.

Let $\kappa : \finitecover{M} \longrightarrow M$ be a finite covering. It is
claimed that $b_1(\finitecover{M})=b_1(M)$. 
\begin{align}
  \label{al:covering}
    \xymatrix{
      \finitecover{\Lambda}_{c} \ar@{^{(}->}[r] \ar@{->>}[d]_{\Kappa|\finitecover{\Lambda}_{c}}&
      T^*\finitecover{M} \ar@{->>}[r]\ar@{->>}[d]^{\Kappa} &
      \finitecover{M} \ar@{->>}[d]^{\kappa} \ar@{->}@/_5mm/[l]_{\finitecover{\eta}_c=\kappa^*\eta_c}\\
      \Lambda_{c} \ar@{^{(}->}[r] &
      T^* {M} \ar@{->>}[r] &
      M \ar@{->}@/^5mm/[l]^{\eta_c}
    }
\end{align}
Since the cotangent lift of $\kappa$, $\Kappa$, is a local
symplectomorphism, the Tonelli Hamiltonian $\finitecover{H}=\Kappa^*H$
is weakly integrable with the first-integral map
$\finitecover{F}=\Kappa^*F$. Let $c \in \cO$ be a cohomology class and
$\eta_c$ a solution to the Hamilton-Jacobi equation for $H$ whose
graph $\Lambda_c$ equals the Mather set $\cM^*_c$ (diagram
\eqref{al:covering}). The pullback
$\finitecover{\eta}_c=\kappa^*\eta_c$ solves the Hamilton-Jacobi
equation for $\finitecover{H}$ and its graph $\finitecover{\Lambda}_c$
is an invariant $C^1$ Lagrangian graph. By Proposition \ref{thm:1},
there is a $\finitecover{d}>0$ such that $\finitecover{\Lambda}_c$
admits the structure of a principal
$\T^{\finitecover{d}}$-bundle. This torus action is defined by
$\finitecover{d}$ commuting vector fields $\finitecover{Y}_i =
X_{\finitecover{F}_i} | \finitecover{\Lambda}_c$,
$i=1,\ldots,\finitecover{d}$ induced by the first-integral map
$\finitecover{F}$. Since $\Kappa$ is a local symplectomorphism,
$\Kappa | \finitecover{\Lambda}_c$ is a local diffeomorphism. This
shows that the dimension $\finitecover{d}$ equals $d$. By the previous
paragraph, weak integrability implies that
$\finitecover{d}=b_1(\finitecover{M})$ so $b_1(\finitecover{M})=b_1(M)$.

Let us prove that $M$ is a trivial principal $\T^d$-bundle. 
This argument is indebted to that of Sepe \cite{2011arXiv1106.4449S}. A
principal $\T^d$-bundle is classified up to isomorphism by a
classifying map
\begin{align}
  \label{eq:classifying-map}
  \xymatrix{
    &
    \T^d \ar@{_{(}->}[dl]\ar@{^{(}->}[d]\ar@{^{(}->}[dr]\\
    M=f^*E\T^d \ar[r]    \ar@{->>}[d]_{\pi_f}&
    E\T^d \ar@{-}[r]\ar@<-.5ex>@{-}[r]   \ar@{->>}[d]_{\pi}&
    \prod_{i=1}^d S^{\infty} \ar@{->>}[d]^{\text{Hopf fib.}}\\
    B \ar[r]^f &
    B\T^d \ar@{-}[r]\ar@<-.5ex>@{-}[r] &
    \prod_{i=1}^d \C{}P^{\infty}.
  }
\end{align}
The classifying map $f$ is null homotopic if and only if  the pullback bundle is
trivial. Classical obstruction theory shows that the single
obstruction to a null homotopy of $f$ is a cohomology class -- the
\textem{Chern class} -- with the following description. The trivial
section $* \mapsto * \times 0$ of $E\T^d$ restricted to its
$0$-skeleton extends over the $1$-skeleton. The obstruction to
extending this section over the $2$-skeleton defines a cohomology
class $\eta \in H^2(B\T^d;\pi_1(\T^d))=H^2(B\T^d;H_1(\T^d))$. By
naturality, the obstruction to extending the trivial section of
$f^*E\T^d$ over the $2$-skeleton is the cohomology class
$\eta_f=f^*\eta \in H^2(B;H_1(\T^d))$ -- called the \textem{Chern
  class}.

In terms of the $E_2$ page of the Leray-Serre spectral sequence with
$\Z$-coefficients for the bundle
$\T^d \hookrightarrow M \longrightarrow B$, one has the differential $d^{0,1}_2 :
E^{0,1}_2=H^1(\T^d) \longrightarrow E^{2,0}_2=H^2(B)$. It has been shown above
that the inclusion map $\T^d \hookrightarrow M$ is injective on $H_1$,
hence surjective on $H^1$. Since a class in $E^{0,1}_2$ survives to a
class in $E_{\infty}$ if and only if  it is in the kernel of $d^{0,1}_2$, the
differential $d^{0,1}_2$ must therefore vanish. Since the differential
$d^{2,0}_2$ vanishes, it follows that $H^2(B)$ survives to
$E_{\infty}$.

On the other hand, for any cohomology class $\phi \in H^1(\T^d)$, the
class $\eta \cup \phi = \kp{\eta}{\phi}$ is a class in $H^2(B\T^d)$
which satisfies $\pi^*(\eta \cup \phi) = 0$ in $H^2(E\T^d)$. By
naturality, the class $\eta_f \cup \phi \in H^2(B)$. This class, if
non-zero, survives to $E_{\infty}$. On the other hand, $\pi_f^*(\eta_f
\cup \phi) = 0$ in $H^2(M)$. This shows that $\eta_f \cup \phi = 0$ in
$H^2(B)$. Since the class $\phi$ was arbitrary, it follows that
$\eta_f$ vanishes. Therefore $M=f^* E\T^d$ is a trivial principal
$\T^d$-bundle.

{
  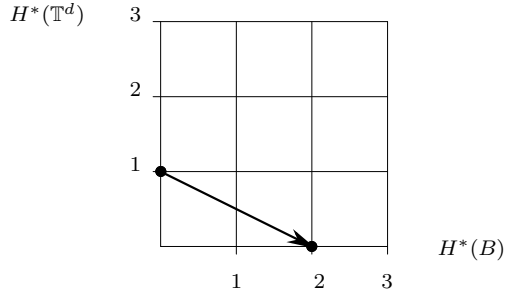
\begin{figure}[htb]
    \centering
    \ifpdf%
    \input{e2-page.pdf_tex}
    \else%
    \input{e2-page.eepic}
    \fi%
    \caption{$E_2$ page of the spectral sequence.}
    \label{fig:spectral-sequence}
  \end{figure}
}

Let us now prove (i--ii).

When $\dim M = 2$, it follows from (i) in Theorem \ref{mainthm} that $M$ is orientable and has genus $0$, therefore it must be $\T^2$.

When $\dim M = 3$, one cannot have $d<3$, since there are no
parallelisable $(3-d)$-dimensional manifolds with trivial first Betti
number. Therefore, $d=3$ and $M = \T^3$.

When $\dim M = 4$, there are two options: $\dim B \leq 2$ or $\dim B = 3$.
When $\dim B \leq 2$, $B$ has the homotopy type of a point, hence it
is a point, so $M = \T^4$. Assume that $\dim B = 3$. If $\pi_1(B)$ is a free product
of irreducible finitely-presented groups $G_i$ ($i=0,\ldots,g$), then
Kneser's theorem \cite{MR2098385} implies that $B =
\connectsum{B_0}{\cdots \connectsum{}{B_g}}$ where
$B_i$ is a closed $3$-manifold with $\pi_1(B_i) = G_i$. Since $H_1(B)
= \bigoplus_i H_1(B_i)$, each homology group $H_1(B_i)$ is
finite. According to \cite[Proposition 2.1]{MR2200344}, if $H_1(B)$ is
finite and $\pi_1(B_i)$ is not perfect for some $i$, then the
universal abelian covering $\abeliancover{B}$, or a $2$-fold cover
thereof, is a finite cover of $B$ which has first Betti number at
least $1$. Thus, the only case to be resolved is that when
$\pi_1(B_i)$ is perfect for all $i=0,\ldots,g$. By \cite[Remark at
bottom of p. 570]{MR2200344}, Stallings' theorem implies that
$G_i=\commutatorsg{G_i}$ is isomorphic to $\pi_1$ of the Klein
bottle -- which is absurd. This proves that $B$ is an irreducible
$3$-manifold. If $\pi_1(B)$ is infinite, then the virtual Haken
conjecture implies that $B$ has a finite covering with non-zero first
Betti number. Therefore, $\pi_1(B)$ is finite and so by the proof of
the Poincar\'e conjecture, $B$ is finitely covered by $S^3$.

Let us prove (iii). Let us denote $\Lambda_{c'}=\{(x,\lambda_{c'}(x)):\; x\in M\}$ as usual. Observe that the map:
\begin{eqnarray*}
\Psi: \cO \times M &\longrightarrow & T^*M \\
(c', x) &\longmapsto& \lambda_{c'}(x)
\end{eqnarray*}
is continuous. It is sufficient to show that if $c_n\rightarrow c'$ in $\cO$, then $\lambda_{c_n}$ converge uniformly to $\lambda_{c'}$. In fact, the sequence $\{\lambda_{c_n}\}_n$ is equilipschitz (it follows from Mather's graph theorem  \cite[Theorem 2]{Mather91}) and equibounded, therefore applying Ascoli-Arzel\`a theorem we can conclude that - up to selecting a subsequence - $\lambda_{c_n}$ converge uniformly to $\tilde{\lambda}= \eta_{c'}+ du$, for some $u\in C^1(M)$. Observe that since $H(x,\lambda_{c_n}(x))=\a_H(c_n)$ for all $x\in M$ and all $n$, and $\a_H$ is continuous, then $H(x,\tilde{\lambda}(x))=\a_H(c')$ for all $x$. Therefore, $u$ is a solution of Hamilton-Jacobi equation $H(x, \eta_{c'}+du)=\a_H(c')$. As we have observed in the beginning of this proof, for each $c'\in\cO$ there is a unique solution of this equation, hence $\tilde{\lambda} = \lambda_{c'}$. This concludes the proof of the continuity of $\Psi$. Notice that this could be also deduced from the fact that $\Psi$ is injective and semicontinuous.

If $\dim H^1(M;\R) \geq \dim M$, then the continuity of $\Psi$ implies that these Lagrangian graphs 
$\Lambda_{c'}$ foliate an open neighborhood of $\Lambda_c$. It follows from 
Proposition \ref{prop2Sor09} that the components of $F$ commute in this open region. Therefore, each $\Lambda_{c'}$ is an $n$-dimensional manifold which is invariant under the action of $n$ commuting vector fields, which are linearly independent at each point. It is a classical result that $\Lambda_{c'}$ is then diffeomorphic to an $n$-dimensional torus and that the motion on it is conjugate to a rotation (see for instance \cite{Arnoldbook}). 

\end{proof}

\section{Amenable groups, measures and rotation vectors}
\label{sec:amenable-group}

In this section it is assumed that $X$ is a compact, path-connected,
locally simply-connected metrizable space and $(G,\haarm[G])$ is a locally
compact, simply-connected, metrizable, amenable topological group with
Haar measure $\haarm[G]$. We will use $d$ to denote a metric on both spaces;
it will be assumed that the metric on $G$ is \textem{right}-invariant,
without loss of generality. The space of $\haarm[G]$-essentially bounded
measurable functions on $G$ is denoted by $\linfty{G}$. $\linfty{G}^*$
has a distinguished subspace of functionals invariant under $G$'s left
(resp. right) action; this subspace will be denoted by
$\linfty{G}^*_{G_-}$ (resp. $\linfty{G}^*_{G_+}$). A functional $\nu
\in \linfty{G}^*$ which satisfies $\nu(1)=1$ is called a
\textem{mean}. The set of left-invariant (resp. right-invariant)
means is denoted by $\means{G}{G_-}$ (resp. $\means{G}{G_+}$);
amenability of $G$ implies that both $\means{G}{G_\pm}$ is non-empty,
as is the intersection $\means{G}{}$.

Let $\abeliancover{\pi} : \abeliancover{X} \longrightarrow X$ be the universal
abelian covering space of $X$, \ie the regular covering space whose
fundamental group is $[\pi_1X, \pi_1X]$ and on which $H_1(X;\Z)$
(singular homology) acts as the group of deck transformations of
$\abeliancover{\pi}$.

Let $\phi : G \longrightarrow X$ be a uniformly continuous map (it is not assumed
that there is an action of $G$ on $X$). The simple-connectedness of
$G$ implies that there is a lift $\abeliancover{\phi}$ of $\phi$ to
$\abeliancover{X}$. It is well-known that the first singular
cohomology group of $X$ is naturally isomorphic to the group of
homotopy classes of maps from $X$ to $S^1$, denoted by $\hone{X}$. For
each $f \in \hone{X}$, let us construct the following commutative
diagram
\begin{align}
  \label{al:hone}
  \xymatrix{
    &
    \abeliancover{X} \ar@{->>}[d]_{\abeliancover{\pi}} \ar[r]|{\abeliancover{f}} &
    \R \ar@{->>}[d]^{p} \\
    G \ar[r]|{\,\phi\,} \ar[ru]|{\,\abeliancover{\phi}\,} \ar@/_4mm/[rr]|{\,g\,} \ar@/^11mm/[urr]|{\,\abeliancover{g}\,} &
    X \ar[r]|f \ar@{..>}[ru]|{\abeliancover{f}} &
    S^1
  }
\end{align}
where $p(x)=x \bmod 1$ and $\abeliancover{f}$ is a lift of $f$ to
$\abeliancover{X}$ --- the dotted diagonal line exists if and only if  $f$ is
null-homotopic. Define the map
\begin{align}
  \label{al:difference-op}
  \xymatrix@R=2mm{
    G \times G \ar[r]^\zeta &
    \R^1 &&
    (s,t) \ar@{|->}[r]^(.4)\zeta &
    g(st)-g(t)\,.
  }
\end{align}
{A priori}, $\zeta$ is a map into $S^1$, but the
simple-connectedness of $G$ implies there is a unique lift of the map in
\eqref{al:difference-op} that is identically zero when $s=1$ (the lift
is trivially $\abeliancover{g}(st)-\abeliancover{g}(t)$). For a fixed
$s \in G$, let $\zeta_s(t) = \zeta(s,t)$.

\begin{lemma}
  \label{lem:linfty}
  For each $s \in G$, $\zeta_s \in \linfty{G}$.
\end{lemma}

\begin{proof}
  Since $X$ is compact, $f$ is uniformly continuous. Since $\phi$ is
  assumed to be uniformly continuous, $g$ and therefore
  $\abeliancover{g}$ is uniformly continuous. Therefore, there is a
  $\delta>0$ such that if $a,b \in G$ and $d(a,b)<\delta$ then
  $|\abeliancover{g}(a)-\abeliancover{g}(b)|<1$. Let $N$ be an integer
  exceeding $d(s,1)/\delta$. Then the right-invariance of the metric
  $d$ implies that for all $t \in G$, $d(st,t)=d(s,1)<N \delta$, so by
  the triangle inequality, one concludes
  $|\abeliancover{g}(st)-\abeliancover{g}(t)|<N$. Thus $|\zeta_s(t)| <
  N$ for all $t \in G$.
\end{proof}

\begin{lemma}
  \label{lem:invariance-of-rotation-vector}
  Let $\nu \in \means{G}{G_-}$ be a left-invariant mean on $G$. If $\abeliancover{g}
  \in \linfty{G}$, then $\kp{\nu}{\zeta_s}=0$ for all $s \in G$. In
  particular, if
  \begin{enumerate}
  \item $f$ is null-homotopic; or
  \item $\im \abeliancover{\phi}$ is contained in a compact set,
  \end{enumerate}
  then $\kp{\nu}{\zeta_s}$ vanishes for all $s \in G$.
\end{lemma}

\begin{proof}
  If $\abeliancover{g} \in \linfty{G}$, then $\kp{\nu}{\zeta_s}=\kp{s_*\nu}{\abeliancover{g}} -
  \kp{\nu}{\abeliancover{g}} = 0$ by left-invariance of $\nu$. If $f$ is
  null-homotopic, then the image of $\abeliancover{f}$ is a compact
  subset of $\R$, so $\abeliancover{g} \in \linfty{G}$; likewise, if $\im
  \abeliancover{\phi}$ has compact closure.
\end{proof}

\begin{lemma}
  \label{lem:dependence-on-phi}
  Let $\phi,\phi' : G \longrightarrow X$ be uniformly continuous maps. If there is
  a $K>0$ such that their lifts
  $\abeliancover{\phi},\abeliancover{\phi}' : G \longrightarrow \abeliancover{X}$
  satisfy $d(\abeliancover{\phi}(s),\abeliancover{\phi}'(s)) < K$ for
  all $s \in G$, then $\kp{\nu}{\zeta_s - \zeta'_s}$ vanishes for all
  $s \in G$ and $\nu \in \means{G}{G_-}$.
\end{lemma}

\begin{proof}
  The proof of this lemma mirrors the preceding. By the assumption
  that $d(\abeliancover{\phi}(t), \abeliancover{\phi}'(t)) < K$ for
  all $t \in G$, one has that $\abeliancover{h}(t) :=
  \abeliancover{f}\abeliancover{\phi}(t) -
  \abeliancover{f}\abeliancover{\phi}'(t)$ lies in
  $\linfty{G}$. Therefore, $\kp{\nu}{\zeta_s - \zeta'_s} =
  \kp{s_*\nu}{\abeliancover{h}} - \kp{\nu}{\abeliancover{h}} = 0$ by
  left-invariance of the mean $\nu$.
\end{proof}

\begin{lemma}
  \label{lem:well-definedness-of-rotation-vector}
  Let $\nu \in \means{G}{G_-}$ be a left-invariant mean and $\phi : G
  \longrightarrow X$ a uniformly continuous map. For each $s \in G$, the map
  \begin{align}
    \label{al:f-map}
    f \longmapsto \kp{\nu}{\zeta_s}
  \end{align}
  (see \eqref{al:difference-op}) induces a linear function
  $\rho_s(\nu) : H^1(X;\R) \longrightarrow \R$. The function $\rho_s :
  \means{G}{G_-} \longrightarrow H_1(X;\R)$ is affine and continuous in the weak-*
  topology on $\linfty{G}^{**}$.
\end{lemma}

\begin{proof}
  It suffices to show that this map is additive on
  $H^1(X;\Z)=\hone{X}$, since it is extended by multiplicativity to a
  map on $H^1(X;\R)$. First, let us show the map is well-defined on
  homotopy classes. Let $f,f'$ be representatives of the homotopy
  class $[f]$. By compactness of $X \times [0,1]$, there is an $N>0$
  such that $|\abeliancover{f}(x)-\abeliancover{f'}(x)|<N$ for all $x
  \in \abeliancover{X}$. Therefore, $|\abeliancover{g}(st)-\abeliancover{g}'(st)|<N$ and
  $|\abeliancover{g}(t)-\abeliancover{g}'(t)|<N$ for all $t \in G$ (using the obvious notation), so
  both $s^*(\abeliancover{g}-\abeliancover{g}')$ and $\abeliancover{g}-\abeliancover{g}'$ are in $\linfty{G}$. Thus,
  $\kp{\nu}{\zeta_s-\zeta'_s} = \kp{s_*\nu}{\abeliancover{g}-\abeliancover{g}'} - \kp{\nu}{\abeliancover{g}-\abeliancover{g}'} =
  0$ by left-invariance of $\nu$. This proves the map \eqref{al:f-map}
  is well-defined on $\hone{X}$.

  To prove that the map \eqref{al:f-map} is additive, let $f,h : X \longrightarrow S^1$
  be representatives of the homotopy classes $[f],[h]$. The homotopy
  class $[f]+[h]$ is represented by $[f+h]$. From the diagram
  \eqref{al:hone}, it is clear that $\zeta^{f+h} = \zeta^f + \zeta^h$
  where $\zeta^{\bullet}$ denotes $\zeta$ constructed with
  $\bullet$. This suffices to prove additivity, and that suffices to
  show that $\rho_s(\nu)$ is a linear form on $H^1(X;\R)$.

  Since the pairing defining $\rho_s(\nu)$ is the bilinear pairing
  between $\linfty{G}^*$ and $\linfty{G}$, it follows that $\rho_s$ is
  an affine map that is continuous in the weak-* topology on linear
  maps $\Hom{\linfty{G}^* ; H^1(X;\R)^*}$.
\end{proof}

\begin{definition}
  \label{def:rotation-set}
  Let $s \in G$. The set
  \begin{align}
    \label{al:rotation-set}
    R_s = \rho_s\left( \means{G}{G_-} \right)
  \end{align}
  is the rotation set of the left translation $s$.
\end{definition}

\begin{theorem}
  \label{thm:rotation-set}
  The map $\rho : G \longrightarrow \Hom{\means{G}{G_-}; H_1(X;\R)}$ is
  continuous. For each $s \in G$, the rotation set $R_s$ is a compact,
  convex subset of $H_1(X;\R)$. The rotation-set map
  \begin{align}
    \label{al:rotation-set-map}
    s \longmapsto R_s
  \end{align}
  is an upper semi-continuous set function.
\end{theorem}

\begin{proof}
  If $s_n \to s$ in $G$, then for a fixed $f : X \longrightarrow S^1$, one sees
  that $\zeta_{s_n} \to \zeta_s$ in $\linfty{G} \cap
  C^0(G;\R)$. Therefore, for any $\nu \in \means{G}{G_-}$,
  $\kp{\rho_{s_n}(\nu) }{ [f] } \longrightarrow \kp{ \rho_s(\nu) }{ [f] }$. This
  proves $\rho$ is continuous in the weak-* topology.

  Clearly $\means{G}{G_-}$ is convex. Since $\means{G}{G_-} \subset
  \linfty{G}^*$ is a closed subset of the unit ball in $\linfty{G}^*$,
  it is a compact set in the weak-* topology. Since $\rho_s$ is
  continuous and affine, its image is compact and convex.
\end{proof}

\subsection{Examples}
\label{sec:examples-rotation-sets}

Let us compute some rotation sets.

\subsubsection{Translations on tori}
\label{sec:translations-on-tori}

Let $X=\T^n$ and let $G=\R^n=\universalcover{X}$ be the universal
covering group acting in the tautological manner; the map $\phi$ is
the orbit map of $\theta_0 \in \T^n$. A cohomology class $f \in
\hone{X}$ has a canonical representative, \viz  $f(\theta) =
\kp{v}{\theta} \bmod 1$ where $v \in \Hom{\Z^n;\Z}$. One arrives at
the map $\universalcover{g}(t) = \kp{v}{t+\universalcover{\theta}_0}$
and $\zeta_s(t) = \kp{v}{s}$ -- which is independent of $t \in G$ --,
whence the mean of $\zeta_s$ equals $\kp{v}{s}$ for any mean $\nu \in
\means{G}{G}$. If one employs the tautological isomorphism between the
real homology (resp. cohomology) group of $\T^n$ and $\R^n$
(resp. $\Hom{\R^n;\R}$), one obtains 
\[
\rho_s(\nu) = s
\]
for all $s \in G$, $\nu \in \means{G}{G_-}$. 

We note that this calculation computes the rotation vector/set of a
subgroup, given a mean on the whole group. Lemma
\ref{lem:iso-rotational} below shows that there is no loss of
generality.

\subsubsection{Translations on quotients of contractible amenable Lie
  groups of \typee{}}
\label{sec:translations-on-conamena-folds}

Let $G$ be a contractible, amenable Lie group of \typee{} (hence a
solvable Lie group of \typee{}), $\Gamma \lhd G$ be a co-compact
subgroup and $X=\Gamma \backslash G$. Let $g,g' \in G$ and let $\phi :
G \longrightarrow X$ be the map $\phi(t)=\Gamma gt^{-1}g'$. Let $N$ be the commutator
subgroup of $G$; it is known that $\Gamma \cap N$ is a lattice in $N$,
that the commutator subgroup of $\Gamma$ is of finite index in $\Gamma
\cap N$ and therefore $\Gamma N$ is closed subgroup of $G$ \cite[Lemma
3]{MR1868577}. The map $F : X \longrightarrow N\backslash X$ is therefore a
submersion onto a torus whose dimension is the codimension of $N$ in
$G$. From the fact that the derived subgroup of $\Gamma$ is of finite
index in $\Gamma \cap N$, one sees that $\hone{X} = F^*\hone{N
  \backslash X}$.

Therefore, we have reduced the problem to the case of a translation on
a torus, whence $\rho_s(\nu) = -Ns$ in the simply-connected abelian
Lie group $N\backslash G$, $\nu \in \means{G}{G_-}$.

\subsubsection{Translations on quotients of amenable Lie groups of \typee{}}
\label{sec:translations-on-amena-folds}

The situation with simply-connected amenable Lie groups of \typee{} is
somewhat more complicated than the previous example, as exemplified by
\cite[Examples 1 \& 2]{MR1868577}. These examples show how the first
Bieberbach theorem may fail, but in these examples the Levi
decomposition is trivial: the groups themselves are solvable and one
might be led to believe that this is the only way that such
pathological examples can arise.

\begin{example*}
Let us give an example where the Levi decomposition is
non-trivial and the first Bieberbach theorem fails. That is, let us
give an example where $G=SK$ is a simply-connected amenable Lie group
of \typee{} where $S$ is its solvable radical and $K$ is a maximal
compact subgroup, and $\Gamma < G$ is a lattice subgroup such that
$\Gamma \cap S$ is not a lattice subgroup of $S$.

Let $k > 2$ be integers and let $N$ be the nilpotent Lie group
whose multiplication is defined by
\begin{align}
  \label{al:hr-multn}
  (x_1,y_1,z_1) \cdot (x_2,y_2,z_2) &= (x_1+x_2,y_1+y_2,z_1+z_2+\frac{1}{2}(x_1 \otimes y_2-x_2 \otimes y_1)) \\
  &\phantom{=} \textrm{where }x_i,y_i \in \R^k, z_i \in \R^k \otimes \R^k.\notag
\end{align}
The cyclic group generated by
\begin{align}
  \label{al:n-cyclic}
  a &=
  \begin{bmatrix}
    2 & 1 & 0\\1 & 1 & 0\\0 & 0 & 1
  \end{bmatrix}
\end{align}
acts as a group of automorphisms of $N$, and this group is a discrete
subgroup of a $1$-parameter group of automorphisms $A$. Let $S=NA$, a
solvable group of \typee{}. On the other hand, let $K$ be the
universal covering group of $\SO{k} \times \SO{k}$ (since $k > 2$,
$K$ is compact) and let $K$ act on $N$ via
\begin{align}
  \label{al:k-action-on-n}
  \kappa \cdot g &= (u\cdot x, v \cdot y, (u \otimes v) \cdot z) \\
  &\phantom{=}\textrm{where }g=(x,y,z) \in N, \kappa\in K \longmapsto (u,v) \in \SO{k} \times \SO{k}.\notag
\end{align}
This is an action by automorphisms of $N$ and this action commutes
with the action of $A$, so this action induces a natural action of $K$
on $S$. This suffices to describe the group $G=SK$, an amenable Lie
group of \typee{}.

The lattice subgroup $\Gamma$ is described as follows. Let
\[
N_{\Z}=\set{ g=(x,y,z) \in N \st x,y \in \Z^k, 2z \in \Z^k
  \otimes \Z^k}
\]
and observe that $a$ preserves $N_{\Z}$. Let $b \in K$ and let $\gamma
= ab$. The group $\Gamma$ generated by $\gamma$ and $N_{\Z}$ is
discrete and co-compact in $G$ for any choice of $b$. If $b$ is of
infinite order, then the intersection of $\Gamma$ with $S$ is just
$N_{\Z}$ and is \textem{not} a lattice in $S$. The projection of
$\Gamma$ to $K=S \backslash G$ is the group generated by $b$; if $b$
is chosen in general position, then the identity component of the
closure is a maximal torus.

This example shows how the first Bieberbach theorem can fail for
\typee{} amenable Lie groups. However, the representation of $K$ as a
group of automorphisms of $S$ is almost faithful, and this implies
many of the nice properties mentioned in the previous paragraph. On
the other hand, if one takes the amenable Lie group $G=\C^n \times
\SU{n}$ with the lattice subgroup $\Gamma$ generated by the set $\set{
  (e_j,\rho_j), (ie_j,\rho_j) \st j=1,\ldots,n}$ where each $\rho_j$
is a generic element in the maximal torus of diagonal matrices, then
one sees that the intersection of $\Gamma$ with $S$ is trivial and the
projection of $\Gamma$ onto $\SU{n}$ is dense in the maximal torus.
\end{example*}

Let $G=SK$ be a simply-connected, amenable Lie group where $S$ is its
radical and $K$ its maximal compact subgroup, and let $\Gamma < G$ be
a lattice subgroup. Let us consider two cases in successive
generality:

\subsubsection*{$K$ is virtually a subgroup of $\Aut{S}$}
In this case, we suppose that the action of $K$ on $S$ by conjugation
has a finite kernel. In this case, the machinery of
\cite{MR1868577,MR0123637} is applicable.

Let $S^*$ be the identity component of the closure of $\Gamma S$ in
$G$ and let $\Gamma^* = S^* \cap \Gamma$. By \cite[Lemma 3]{MR1868577}
and \cite{MR0123637}, one knows that $S^*$ is a solvable subgroup
containing $S$, $\Gamma^*$ is of finite index in $\Gamma$, $S
\backslash S^*$ is a torus subgroup, $T$, of $K$, the nilradical of
$S^*$ equals the nilradical $R$ of $S$, $\Gamma \cap R$ is a lattice
subgroup of $R$. Likewise, the derived subgroup of $S$,
$N=\commutatorsg{S}=\commutatorsg{S^*}$, intersects $\Gamma$ in a
lattice subgroup of $N$. This information is summarised in the
commutative diagram \eqref{al:comm-dia-amgp}, 
where $\Beta=\Gamma \cap N, F=\Beta \backslash N, Z=\Beta \backslash
G$, $T^*=(\Beta \backslash N) \backslash (N \backslash S^*)$ (a
torus) and $A=S \backslash S^*$.

\begin{align}
  \label{al:comm-dia-amgp}
  \xymatrix@C=2.5mm@R=1.5mm{
    &&&
    \Beta \ar@{->}[rrr] \ar@{->}[ddd] \ar@{->}[rd] &&&
    \Beta \ar@{->}'[d]'[dd][ddd] \ar@{->}[rd]
    \\
    &&& &
    N \ar@{->}[rrr] \ar@{->}[ddd] \ar@{->>}[rd] &&&
    G \ar@{->>}[rrr] \ar@{->}'[d][ddd] \ar@{->>}[rd] &&&
    N \backslash G = A K \ar@{->>}[ddd]
    \\
    &&& &&
    F \ar@{->}[rrr] \ar@{->}[ddd] &&&
    Z \ar@{->>}[ddd] \ar@{->>}[urr]
    \\
    \Beta \backslash \Gamma^* \ar@{->}[rd] &&&
    \Gamma^* \ar@{->}'[r]'[rr][rrr] \ar@{->}'[d]'[dd][ddd] \ar@{->}[rd] \ar@{->>}[lll] &&&
    \Gamma^* \ar@{->}'[d]'[dd][ddd] \ar@{->}[rd]
    \\
    &
    N \backslash S^* \ar@{->>}[rd] &&&
    S^* \ar@{->}'[r][rrr] \ar@{->}'[d][ddd] \ar@{->>}[rd] \ar@{->>}[lll] &&&
    G   \ar@{->>}'[r][rrr] \ar@{->}'[d][ddd] \ar@{->>}[rd] &&&
    S^* \backslash G = T \backslash K \ar@{->>}[ddd]
    \\
    &&
    T^* &&&
    Y^* \ar@{->}[rrr] \ar@{->>}[ddd] \ar@{->>}[lll] &&&
    X^* \ar@{->>}[ddd] \ar@{->>}[urr]
    \\
    &&&
    \Gamma \ar@{->}'[r]'[rr][rrr] \ar@{->}[rd] &&&
    \Gamma \ar@{->}[rd]
    \\
    &&& &
    \overline{\Gamma S} \ar@{->}'[r][rrr] \ar@{->>}[rd] &&&
    G   \ar@{->>}'[r][rrr] \ar@{->>}[rd] &&&
    \overline{\Gamma S} \backslash G = WT \backslash K
    \\
    &&& &&
    Y \ar@{->}[rrr] &&&
    X \ar@{->>}[urr]
  }
\end{align}
In diagram \eqref{al:comm-dia-amgp},
all southeast sequences are fibrations with discrete fibre (covering
spaces), all eastern sequences are fibrations, as are the backwards
$L$ sequences. In particular, $X^*$ is a finite regular covering space
of $X$ which is fibred by the solvmanifold $Y^*$ over the
$K$-homogeneous space $T \backslash K$; the solvmanifold $Y^*$ is
itself fibred by the nilmanifold $F$ over the torus $T^*$. Since $S^*$
is the identity component of $\overline{\Gamma S}$, the group
$W=\Gamma^* \backslash \Gamma$ permutes the components of
$\overline{\Gamma S}$, which shows that $Y^*=Y$, so $X$ is fibred by
solvmanifolds, also.

Since $WT \backslash K$ has finite fundamental group, its first
cohomology group over $\Z$ vanishes. Therefore, the Leray-Serre
spectral sequence for the fibring of $X$ by $Y$ shows that the
restriction to a fibre induces an injection of $H^1(X;\R)$ into
$H^1(Y;\R)$ (the image is the kernel of $d_2^{0,1}$ in the figure
\ref{fig:spectral-sequence}). The fibring of $Y$ by the nilmanifold
$F$ over the torus $T^*$ is exactly as described in the previous
example. In particular, the projection map induces an isomorphism of
$H^1(Y;\R)$ and $H^1(T^*;\R)$. Since $S^*=ST$, we see that $N
\backslash S^* = AT$ where $A=N \backslash S$. Since $T$ is
contractible in $G$, one sees that the first real homology group of
$X^*$ is naturally identified with $A$; or $Z$ is visibly the
universal abelian covering space of $X^*$. It follows that $H^1(X;\R)$
is naturally identified with $A^W$, the fixed-point set of $W$ acting
on $A$.

Let $\phi : G \longrightarrow X$ be defined by $\phi(t) = \Gamma g t^{-1} g'$ for
some $g,g' \in G$. A few applications of Lemma
\ref{lem:dependence-on-phi} imply that one can suppose, without
changing the rotation map, that $\phi(t) = \Gamma a \kappa \alpha^{-1}
\kappa^{-1} b$ where $a,b \in S$, $\kappa \in K$ and $t=\beta \alpha$
is the decomposition into $\beta \in K$ and $\alpha \in S$. Let
$\abeliancover{F} : Z \longrightarrow A=N \backslash G/K$ be the map that induces
the isomorphism of $\hone{X^*} \otimes \R$ with $A$. Concretely, if
$Nt \in Z$, let $t=\beta_t \alpha_t$ be the decomposition of $t$ into
$\beta_t \in K$, $\alpha_t \in S$; then $\abeliancover{F}(Nt)=K
\alpha_t N$. One computes that
\begin{align}
  \label{al:zeta-amgp}
  \zeta_s(t) &= -K (\kappa \beta_t^{-1}) \cdot \alpha_s \cdot (\kappa \beta_t^{-1})^{-1} N
  &&
  s,t \in G.
\end{align}
It is clear that $\zeta_s$ is $S$-invariant since $t \mapsto \beta_t$
is the projection $G \longrightarrow K$. Since the restriction of any mean on a
compact Lie group to its continuous functions is the Haar probability
measure \cite{MR961261}, one sees that for any $\nu \in
\means{G}{G_-}$, $\rho_s(\nu)=-\bar{\alpha}_s N$ is the projection of
$\alpha_s N$ onto the subspace of $K$-invariant vectors.

Note that if one restricts $\phi$ to $S$, then the rotation vector of
$s \in S$ with respect to the mean $\nu \in \means{S}{S_-}$ is the
projection of $-\kappa s \kappa^{-1} N$ onto the subspace of
$W$-invariant vectors.

\subsubsection*{When $K$ is not a virtual subgroup of $\Aut{S}$}
Let us now examine the case where the kernel of representation $K \longrightarrow
\Aut{S}$ is not finite. Let $K_1 \lhd K$ be the identity component of
this kernel. Since $K$ is compact and simply-connected, $K$ is
semi-simple and so $K=K_0 \oplus K_1$ is a sum of semi-simple factors,
and the representation of $K_0 \longrightarrow \Aut{S}$ has finite kernel. By
construction, $K_1$ is a normal subgroup of $G$ and the lattice
$\Gamma$ intersects $K_1$ in a compact set, hence $\Gamma \cap K_1$ is
a finite, normal subgroup of $\Gamma$. We obtain the fibration
\begin{align}
  \label{al:not-virtual-auts-fibration}
  \xymatrix{
    \Gamma \cap K_1 \backslash K_1 \ar[r] &
    \Gamma \backslash G \ar@{->>}[r]|{\,\rho\,}  &
    \bar{\Gamma} \backslash \bar{G} & = (\Gamma \cap K_1 \backslash \Gamma) \backslash (K_1 \backslash G)\,.
    }
\end{align}
The quotient $\bar{G}=SK_0$ has the property that $K_0$ is a virtual
subgroup of $\Aut{S}$. The fibre $\Gamma \cap K_1 \backslash K_1$ has
a finite fundamental group. It follows that the map $\rho^* :
H^1(\bar{\Gamma} \backslash \bar{G}; \R) \longrightarrow H^1(\Gamma \backslash G;
\R)$ is an isomorphism. From this, one concludes that the preceding
computations of the $\zeta$-map \eqref{al:zeta-amgp} and the rotation
vectors of a mean remain correct in this enlarged setting.

\subsubsection{Quotients of amenable Lie groups of \typee{} -- II}
\label{sec:translations-on-amena-folds-ii}

Let us continue with the notations of the previous example. Let $H=G
\times G'$ be a product of simply-connected amenable Lie groups (in
applications, $G'=\R$, but what follows is perfectly general). Let
$\varphi : G' \longrightarrow X$ be a uniformly continuous map and let
\begin{align}
  \label{al:amengp-map-x2}
  \phi &: H \longrightarrow X &
  \phi(h) &= \Gamma g^{-1} \varphi(g'),
  \textrm{ where } h=(g,g') \in H.
\end{align}
Similar to that above, one computes that with $s=(1,b)$ and $t=(g,a)$,
one has
\begin{align}
  \label{al:amengp-zeta-s-x2}
  \zeta_s(t) &= -K \delta_b(a) N &&
  \delta_b(a)=\textrm{the projection of }\varphi(ba)^{-1} \cdot \varphi(a)\textrm{ onto }S,
\end{align}
using the factorisation of an element in $G$ as in the previous
example. In particular, this implies that $\zeta_s$ is independent of
$g$ when $s=(1,b)$. This implies that if $\nu \in \means{H}{H_-}$ is a
mean on $H$, then the rotation vector $\rho_s(\nu)$ ($s=(1,b)$) equals the
rotation vector $\rho_b(\bar{\nu})$ for the map $\varphi$ and the
projected mean $\bar{\nu} \in \means{G'}{G'_-}$.

In the next section we show how this result can be interpreted in
terms of the rotation vector of two measures with differently-sized
supports.

\subsection{Relation to Schwartzman cycles}
\label{sec:relation-to-schwartzman}

Let us suppose that $\Phi : G \times X \longrightarrow G$ is a left-action of $G$
on $X$. For each $x \in X$, one has the orbit map $\phi_x(t) :=
\Phi(t,x)$. The action will also be denoted by $\Phi(t,x) = t \cdot
x$.

\begin{lemma}
  \label{lem:uniform-continuity-of-orbit-map}
  The orbit map $\phi_x : G \longrightarrow X$ is uniformly continuous for all $x
  \in X$.
\end{lemma}
\begin{proof}
  Let us define $\epsilon(\delta) = \max \set{ d(\Phi(1,x),\Phi(t,x))
    \st x \in X, d(1,t) \leq \delta}$. By local compactness of $G$ and
  compactness of $X$, the maximum is attained. Moreover, $\epsilon$ is
  a continuous increasing function of $\delta$ that vanishes at
  $\delta=0$. This implies uniform continuity of the orbit map
  $\phi_x$.
\end{proof}

Let $\nu \in \means{G}{G_-}$ be a left-invariant mean on $G$. For each
$x \in X$, the pull-back of $C^0(X)$ by the orbit map $\phi_x$ lies
inside $\linfty{G}$. Thus, $\phi_{x,*} \nu$ determines a positive,
continuous linear functional on $C^0(X)$ and so by the Riesz
representation theorem, $\phi_{x,*} \nu$ induces a Borel probability
measure $\mu_x$ on $X$. It is clear that $\mu_x$ is $G$-invariant. The
support of $\mu_x$ is clearly contained in the $\omega$-limit set of
$x$,
\begin{align}
  \label{al:omega-limit-set}
  \omega_G(x) &= \bigcap_{T>0} \overline{\set{ t\cdot x \st d(1,t)>T}}.
\end{align}

In \cite[Appendix A]{FGS}, one finds a definition of the rotation
vector of an invariant measure of a flow (an $\R$-action). Let $\mu$
be an invariant Borel probability measure of the flow $\varphi : \R
\times X \longrightarrow X$ and $[f] \in \hone{X}$ a cohomology class. The
rotation vector of $\mu$ is defined as
\begin{align}
  \label{al:fgs-rotation-vector}
  \kp{[f]}{\rho_{\varphi}(\mu)} = \int_{x \in X} \zeta_{\varphi}(x)\, d\mu(x),
\end{align}
where $\zeta_{\varphi}(x)=f(\varphi_1(x)) - f(x)$ similar to
\eqref{al:difference-op}. We have:

\begin{theorem}
  \label{thm:fgs-equality}
  Let $\Phi : G \times X \longrightarrow X$ be a $G$-action, $\varphi$ be an
  action of a $1$-dimensional subgroup with $\varphi_1=s$, and let
  $\nu \in \means{G}{G_-}$, $\mu_x = \phi_{x,*} \nu$ for some $x \in
  X$. Then
  \begin{align}
    \label{al:rho}
    \rho^x_s(\nu) = \rho_{\varphi}(\mu_x),
  \end{align}
  where $\rho^x$ is the rotation map for the orbit map $\phi_x$.
\end{theorem}

The proof is an application of change of variables.

\subsection{Averaged rotation vectors}
\label{sec:averaged-rotation-vectors}

In this subsection, let us suppose that $G$ fits in the exact sequence
of (amenable) groups
\begin{align}
  \label{al:g-exact-seq}
  \xymatrix{
    H \ar@^{{(}->}[r] &
    G \ar@{->>}[r] &
    F.
  }
\end{align}
Let $\nu_H \in \means{H}{H_-}$ (resp. $\nu_F \in \means{F}{F_-}$) be
left-invariant means. One can define an invariant mean $\nu_G$ as
follows: let $f \in \linfty{G}$ and define $f_H \in \linfty{F}$ by
averaging over $H$, $f_H(Ht) = \kp{\nu_H}{f_t}$ where $f_t(x) =
f(tx)$. The normality of $H$ and left-invariance of $\nu_H$ implies
that $f_H$ is well-defined and $f_H \in \linfty{F}$. Then, one defines
the left-invariant mean $\nu_G$ by $\kp{\nu_G}{f} := \kp{\nu_F}{f_H}$.

\begin{definition}
  \label{def:product-mean}
  The mean $\nu_G \in \means{G}{G_-}$ is denoted by $\nu_G = \nu_F
  \times \nu_H$ and called a \textem{product mean}.
\end{definition}

Let us suppose that $H$ acts on $X$ by an action $\varphi$ and that
there is a uniformly continuous map $\phi : G \longrightarrow X$ satisfying
\begin{align}
  \label{al:twisted-action}
  \phi(s\cdot t) &= \varphi(s) \cdot \phi(t) && \forall s\in H, t\in G.
\end{align}

Let $t_0 \in G$, $x = \phi(t_0)$ and $\mu_{H,x} = \varphi_{x,*} \nu_H$
is the pushed forward measure on $X$. The measure $\mu_G = \phi_*
\nu_G$ (where $\nu_G = \nu_F \times \nu_H$) is $H$-invariant due to
the cocycle condition \eqref{al:twisted-action} and $\supp \mu_{H,x}
\subset \supp \mu_G$.

The following lemma shows that under a suitable condition on the map
$\phi$, one can average over the group $G$ to obtain a measure $\mu_G$
with a larger support and the same rotation set.

\begin{lemma}
  \label{lem:iso-rotational}
  Suppose that the lift $\abeliancover{\phi}$ (see \eqref{al:hone})
  has the property that for each $t \in G$, there is a $K>0$ such that
  $d( \abeliancover{\varphi}(s) \cdot \abeliancover{\phi}(t_0),
  \abeliancover{\varphi}(s) \cdot \abeliancover{\phi}(t) ) < K$ for
  all $s \in H$. Then for all $s \in H$, $\rho_s(\mu_{H,x})$ is
  independent of the point $x \in \im \phi$. In particular,
  \begin{align}
    \label{al:iso-rotational-equality}
    \rho_s(\mu_{H,x}) = \rho_s(\mu_G).
  \end{align}
\end{lemma}
To be clear, $\rho_s$ refers to the rotation map of the flow generated
by the $1$-parameter group through $s$, as in
\eqref{al:fgs-rotation-vector}. The proof of this lemma follows from
Lemma \ref{lem:dependence-on-phi} and Theorem \ref{thm:fgs-equality}
along with an unraveling of the product mean.

Note that the example in section \ref{sec:translations-on-amena-folds}
does not contradict this lemma. In that example, the map $\phi$ does
not satisfy the uniform boundedness condition.

\section{Homogeneous structures}
\label{sec:homogeneous-structures}

This section proves Theorems \ref{thm:2} and
\ref{thm:2-implications}. We begin by establishing some terminology
and notation.

Let $G$ be a connected Lie group. Define the left (resp. right)
translation map by
\begin{align}
  \label{eq:lr-actions}
  L_h(g) &:= hg, && R_h(g) := gh
\end{align}
for all $g,h\in G$. These two maps define a left action of $G_-=G$
(resp. $G_+=G^{op}$) on $G$ and therefore on $T^*G$ by Hamiltonian
symplectomorphisms. The momentum maps of these actions are
\begin{align}
  \label{al:1}
  \mmap_- &: T^*G \longrightarrow \g_-^*  &
  \mmap_+ &: T^*G \longrightarrow \g_+^*  \\
  \mmap_-(g,\mu_g) &:=  (T_1 R_g)^* \mu_g &
  \mmap_+(g,\mu_g) &:=  (T_1 L_g)^* \mu_g, \notag
\end{align}
for each $g\in G$, $\mu_g \in T^*_g G$.

A co-vector field $\mu : G \longrightarrow T^*G$ is left- (resp. right-) invariant
if $\mu(1) = (T_1 L_g)^* \mu(g)$ (resp. $\mu(1) = (T_1 R_g)^* \mu(g)$)
for all $g \in G$. If one trivialises $T^*G$ with respect to the
left-invariant co-vectors, then the momentum maps are simply
\begin{align}
  \label{al:2}
  \mmap_-(g,\mu) &:= \CAd{g^{-1}} \mu &&
  \mmap_+(g,\mu)  := \mu,
\end{align}
for all $g \in G, \mu \in \g^* = T_1^* G$, where $\CAd{g} = (T_1
L_gR_{g^{-1}})^*$.

One says that a function $H : T^*G \longrightarrow R$ is \textem{collective} for
the left-action (resp. right-action) if $H=\mmap_-^*h$
(resp. $H=\mmap_+^*h$) for some $h : \g^* \longrightarrow \R$. If $H$ is
collective for the left-action (resp. right-action) then \eqref{al:1}
shows it is right-invariant (resp. left-invariant). In particular, a
Hamiltonian that is collective for the left-action [right-invariant]
(resp. right-action [left-invariant]) Poisson-commutes with $\mmap_+$
(resp. $\mmap_-$).

Let $H : T^*G \longrightarrow \R$ be a smooth, left-invariant (= right collective)
Tonelli Hamiltonian. Therefore, there is a smooth convex Hamiltonian
$h : \g^* \longrightarrow \R$ such that $H=\mmap_+^*h$. Moreover, since $H$ is
left-invariant, it Poisson-commutes with the momentum map of the left
action $\mmap_-$. 

Let $\Gamma \lhd G$ be a co-compact lattice subgroup and $M=\Gamma
\backslash G$. It is assumed that $G$ is simply-connected, so that the
universal cover of $M$, $\universalcover{M}$, is $G$. Let
$\commutatorsg{\Gamma}=\commutator{\Gamma}$ be the commutator subgroup
of $\Gamma$, which is the fundamental group of the universal abelian
cover $\abeliancover{M}$. This leads to the commuting diagram of
covering maps:
\begin{equation}
  \label{eq:coverings}
  \xymatrix{
    T^* G=T^*\universalcover{M} \ar@{->>}[r]\ar@{->>}[d]^{\abeliancover{\Pi}} \ar@{->>}@/_17mm/[dd]_\Pi &
    G=\universalcover{M} \ar@{->>}[d]^{\universalcover{\pi}} \ar@{->>}@/^17mm/[dd]^\pi\\
    T^* (\commutator{\Gamma} \backslash G)=T^* \abeliancover{M} \ar@{->>}[r] \ar@{->>}[d]^{\abeliancover{\Pi}}&
    \commutator{\Gamma} \backslash G=\abeliancover{M} \ar@{->>}[d]^{\abeliancover{\pi}}\\
    T^* ({\Gamma} \backslash G)=T^* {M} \ar@{->>}[r] &
    {\Gamma} \backslash G={M}\,.
    }
\end{equation}
We adopt the notational convention that the pull-back of $x$ to
$\abeliancover{M}$ (resp. $\universalcover{M}$) is denoted by
$\abeliancover{x}$ (resp. $\universalcover{x}$).

Let $c \in H^1(M;\R)$ be a cohomology class, let $(x,p) \in
\cM_c^*(H)$ be a recurrent point in the Mather set and let
$\delta : \R \longrightarrow M$ be the minimizer with initial conditions
$\delta(0)=x$ and $\cL (x,\dot{\delta}(0))=(x,p)$, where $\cL$ denotes the associated Legendre transform (see section \ref{sec:am-theory}). By the arguments of
\cite{Mather91}, we can suppose that the rotation set of $\delta$ is a
singleton $\set{h} \subset H_1(M;\R)$ and any weak-* limit of uniform
measures along the orbit is a minimizing measure. Fix a lift
$\universalcover{\delta}$ of $\delta$ to $\universalcover{M}$. For
each $g \in G$, let $\universalcover{\delta}_g = L_{g^{-1}} \circ
\universalcover{\delta}$ be a left-translate of this lift. Left
invariance of $H$ implies that $\universalcover{\delta}_g$ is the
projection of an integral curve, which implies that the projection of
$\universalcover{\delta}_g$ to $\abeliancover{M}$ and $M$ are also
projections of orbits. All of this allows the definition of a map
\begin{align}
  \label{al:phi-map-mather-set}
  \xymatrix{
    G \times \R \ar[r]^{\phi} \ar[rd]_{\bar{\phi}} &
    T^*M        \ar[d] &&
    \phi(g,t) = \Pi \circ (TL_{g^{-1}})^* \universalcover{\varphi}_t(x,p), \\
    &
    M &&
    \bar{\phi}(g,t) = \Gamma g^{-1} \tilde{\delta}(t) = \Gamma \delta_g(t)
    }
\end{align}
where $\universalcover{\varphi}$ is the flow of $H$ on the universal
cover $T^*G$. By the example in section
\ref{sec:translations-on-amena-folds-ii}, the rotation vector of the
map $\delta_g$ is independent of $g$ for any mean on $G \times
\R$. This implies the same is true for $\phi(g,t)$.

Let $\nu_{\R} \in \means{\R}{\R}$ be an invariant mean such that the
rotation vector of $\nu_{\R}$ at $s=1$ under the map $\delta$ is
$h$. By hypothesis, there is such a mean. The preceding discussion proves the following Lemma.

\begin{lemma}
  \label{lem:min}
  Let $\nu_\R \in \means{\R}{\R}$, $\nu_G \in \means{G}{G_-}$ and
  $\mu=\phi_* (\nu_\R \times \nu_G)$. Then $\mu$ minimizes $A_c$ - \ie it is $c$-action minimizing - and
  the projection of $\supp \mu$ covers $M$.
\end{lemma}

\begin{proof}[Proof (Theorem \ref{thm:2})]
By \cite[Theorem 2]{Mather91}, we know that $\supp \mu$ is
a Lipschitz graph over $M$. Therefore, the lift to $T^*
\universalcover{M}$ contains the smooth manifold
$\universalcover{\phi}(G \times 0)$ which is a smooth graph over
$\universalcover{M}$. Therefore $\supp \mu$ is a smooth Lagrangian
graph over $M$, $\supp \mu = \Graph{\eta}$, and lifting this
picture to $T^* \universalcover{M}$ shows that
$\universalcover{\eta}$ is closed and left-invariant. Therefore,
$\universalcover{\eta}$ is a bi-invariant $1$-form. Since $\cM^*_c(H)
= \Graph{\eta}$, where $c$ is the cohomology class of $\eta$, this
proves item (i) of Theorem \ref{thm:2}. Lemma \ref{lem:min} and the
preceding discussion implies that the rotation set of $\cM^*_c(H)$ is
a singleton, which implies item (iii).

Let us now examine Hamilton's equations for $H$ on the Mather set
$\cM^*(H) = \Graph{\eta}$.
Since $H$ is left-invariant, it follows that 
\begin{align}
  H(q,\eta(q)) &= h \circ \mmap_+(q,\eta(q)) = h((T_1 L_g)^* \eta(q)) = h(\eta(1)) = E   \label{al:3}\\
  \mmap_-(q,\eta(q)) &= \eta(1), \label{al:4}
\end{align}
for all $q\in M$. (\ref{al:4}) follows because $\eta$ is bi-invariant,
which implies that the co-adjoint orbit of $\eta(1)$ is a single
point.

    Hamilton's equations for the Hamiltonian $H$ are
    \begin{align}
      X_H(g,\mu) &: \left\{
        \begin{array}{lcl}
          \dot{g}   &=& (T_1 L_g)\cdot \d h(\mu),\\
          \dot{\mu} &=& -\cad{\d h(\mu)} \mu.
        \end{array}
      \right.
      && \forall g\in G, \mu \in \g^*.
    \end{align}
    In particular, if $\mu$ is a closed form, then $\cad{\xi}\mu$
    vanishes for all $\xi \in \g$. Therefore, the orbit of $(g,\mu)$ is
    $\{ (T_1 R_{\exp(t\xi)})^* (g,\mu) = (g \exp(t\xi),\mu)\ : \ t \in
    \R \}$ where $\xi=\d h(\mu)$, \ie it is the orbit of a
    $1$-parameter subgroup. This proves item (ii).

    Finally, the discussion around \eqref{eq:hj-inequation} and
    \eqref{al:3} shows that the following diagram commutes
    \begin{align}
      \label{eq:alpha-H-amen}
      \xymatrix{
        H^1(M;\R)  \ar[r]^(.7){\alpha_H} & \R  &&
        c \ar@{|->}[r]                   & \alpha_H(c) = h(\eta(1)) \\
        (T^*_1G)^G \ar[u]^{\cong} \ar[ur]_{h} &&&
        \eta \ar@{|->}[u] \ar@{|->}[ur]
      }
    \end{align}
    where $(T^*_1G)^G$ is the set of bi-invariant $1$-forms on $G$. By
    hypothesis, $h$ is $C^r$. This completes the proof.
\end{proof}

\begin{proof}[Proof (Theorem \ref{thm:2-implications})]
  The sole remaining thing to prove is that if $H$ is
  weakly integrable and $\Lambda \subset T^*M$ lies inside an
  iso-energy surface and intersects $\reg{F}$, then $M$ is a
  homogeneous space of a compact reductive Lie group. By Theorem
  \ref{thm:2}, $\Lambda = \Graph{\eta}$ where $\eta$ is a
  bi-invariant, closed $1$-form on $G$. By Theorem \ref{maincor},
  $M$ is diffeomorphic to $\T^b \times B$ where $B$ is a
  parallelisable manifold with finite coverings having zero first
  Betti number. Therefore, the lattice $\Gamma = \pi_1(M)$ splits as
  $\Gamma = \Z^b \oplus P$ where $P=\pi_1(B)$. From the description in 
 (\ref{al:comm-dia-amgp}), one knows that $\Beta$ and hence $N$ must be
  trivial. This implies that $\dim S=b$ (we do not claim that the
  $\Z^b$ factor is a lattice in $S$). On the other hand, one also sees
  that $P=\pi_1(B)$ must be finite: since $\Gamma$ is virtually
  polycyclic, so is $\Z^b \backslash \Gamma = P$, but a virtually
  polycyclic group is either finite or it contains a finite index
  subgroup that has non-zero first Betti number.
  \footnote{If $D$ is solvable, then the derived series
    $D_k=\commutatorsg{D_{k-1}}$, $D_0=D$, terminates at $1$ for some
    $k$. If each quotient $D_{k-1}/D_k$ is finite, then $D$ is finite;
    if $D$ is not finite, then there is a least $k$ such that
    $D_k/D_{k+1}$ is infinite. This $D_k$ is therefore of finite index
    with non-zero first Betti
    number.}
  Additionally,   since $P < G$ is a finite subgroup, it is compact and therefore a
  subgroup of a maximal compact subgroup; up to an inner automorphism,
  we can assume that $P < K$.

  Therefore $M$ is finitely covered by $\abeliancover{M}=\T^b \times
  \universalcover{B}$ and Theorems~\ref{maincor} \&~\ref{thm:2} show that
  $\T^b$ is the closure of the projection of a $1$-parameter subgroup
  of $S$. This proves that $S$ is abelian.

  Finally, let $\Gamma_1 < \Gamma$ be a torsion-free subgroup such
  that $\abeliancover{M} = \Gamma_1 \backslash G$. One knows that
  $\Gamma_1$ is generated by elements $\epsilon_i=e_i \delta_i$ for
  $i=1,\ldots,b$ where $e_i \in S$, $\delta_i \in K$. Since $\Gamma_1$
  is abelian, the $\delta_i$ pairwise commute and $e_i$ commutes with
  $\delta_j$ for all $i \neq j$. From the argument of section
  \ref{sec:translations-on-amena-folds} one knows that there are
  integers $n_i>0$ such that $\delta_i^{n_i}$ generate a torus
  subgroup $T<K$. It follows that there are torsion elements $c_i \in
  K$ and $\xi_i \in \liealg{T}$ such that $\delta_i = c_i\exp(\xi_i)$
  and the $c_i$ pairwise commute and commute with all $\delta_j$. Let
  us define $\epsilon_{i,t}=e_i c_i \exp(t\xi_i)$ and $\Gamma_t$ be
  the lattice subgroup of $G$ generated by $\epsilon_{i,t}$. The
  identity map on $G$ induces a diffeomorphism of $\Gamma_0 \backslash
  G$ with $\Gamma_1 \backslash G = \universalcover{M}$. The lattice
  $\Gamma_0$ is generated by $\epsilon_{i,0} = e_ic_i$. Since
  $\Gamma_0$ is abelian, the $c_i$ must fix each $c_j$, $j\neq i$, and
  $c_i$ must send $e_i$ to $\pm e_i$. If $c_i e_i c_i^{-1} = -e_i$,
  then $\epsilon_{i,0}$ is a torsion element in the free abelian group
  $\Gamma_0$, hence it is $1$, absurd. Therefore, $c_i$ fixes $e_i$,
  too. Since $\set{e_i}$ generates a lattice in $S$, each $c_i$
  commutes with $S$. Therefore, $c_i \in \ker(K \longrightarrow \Aut{S})$ for each
  $i$.

  To sum up: let $\Gamma_0^t \lhd \Gamma_0$ be the sublattice
  generated by the pure translations in $\Gamma_0$. Then $\Gamma_0^t
  \backslash G$ is diffeomorphic to $\T^b \times K$, a reductive Lie
  group and it is a smooth covering space of $M$.
\end{proof}

\bibliographystyle{amsplain}
\bibliography{weak-integrability}
\end{document}

%% file: definitions.tex
\usepackage{amssymb,graphicx,amsthm}
\usepackage[all]{xy}
\xyoption{import}
\usepackage{color,ifthen}
\usepackage{eepic}
\usepackage{comment}
\usepackage{hyperref}
\usepackage[utf8]{inputenc}

\newcommand{\Z}{ {\mathbb Z}   }
\newcommand{\N}{ {\mathbb N}   }

\newcommand{\C}{ {\mathbb C}   }
\newcommand{\T}{ {\mathbb T}   }
\newcommand{\R}{ {\mathbb R}   }
\newcommand{\liealg}[1]{{\rm Lie}\, #1\,}
\newcommand{\lie}[1]{\mathfrak{#1}}
\renewcommand{\t}{\lie{t}}
\newcommand{\g}{\lie{g}}

\newcommand{\SO}[1]{{\rm SO}_{#1}}
\newcommand{\SU}[1]{{\rm SU}_{#1}}
\newcommand{\Aut}[1]{\textrm{Aut}(#1)}

\newcommand{\Graph}[1]{\mathrm{graph}(#1)}

\newcommand{\lp}[2]{L^{#1}(#2)}
\newcommand{\linfty}[1]{\lp{\infty}{#1}}
\newcommand{\means}[2]{\mathfrak{m}(#1)^{#2}}
\newcommand{\hone}[1]{[#1,S^1]}
\newcommand{\im}[1]{\mathrm{Im}\,#1}
\newcommand{\Hom}[1]{\mathrm{Hom}(#1)}
\newcommand{\haarm}[1][]{m_{#1}}

\newcommand{\textem}[1]{{\em #1}}
\newcommand{\reg}[1]{{\rm Reg}\, #1}
\newcommand{\mmap}{\Psi}
\newcommand{\CAd}[1]{{\rm Ad}^*_{#1}}

\newcommand{\cad}[1]{{\rm ad}^*_{#1}}

\renewcommand{\d}{{\rm d}}

\newcommand{\ie}[0]{\hbox{\it i.e.\ }}
\newcommand{\interalia}[0]{\hbox{\it inter alia\ }}
\newcommand{\viz}[0]{\textit{viz.}}
\newcommand \rank {{\rm rank\ }}
\newcommand \supp {{\rm supp\ }}
\newcommand{\set}[1]{\left\{ #1 \right\}}
\newcommand{\universalcover}[1]{\tilde{#1}}
\newcommand{\abeliancover}[1]{\hat{#1}}
\newcommand{\finitecover}[1]{\check{#1}}

\newcommand{\commutator}[1]{#1_{1}}
\newcommand{\commutatorsg}[1]{[#1,#1]}
\newcommand{\Kappa}[0]{\mathcal{K}}
\newcommand{\st}[0]{\,:\,}

\newcommand{\typee}[0]{type (E)}
\newcommand{\Beta}[0]{\textrm{B}}
\newcommand{\cprime}{$'$}
\newcommand{\Arnold}[1]{\ifthenelse{\equal{#1}{}}{Arnol\cprime{}d}{ARNOL\cprime{}D}}

\newcommand{\cA}{\mathcal{A}}
\newcommand{\cM}{\mathcal{M}}
\newcommand{\cN}{\mathcal{N}}
\newcommand{\cO}{\mathcal{O}}

\newcommand \cL {{\mathcal L}}
\newcommand{\calM}{\mathfrak{M}}
\newcommand \cS {{\mathcal S}}
\renewcommand \a {\alpha}

\newcommand\beqa[1]{ \begin{eqnarray} \label{#1}}
\newcommand{\eeqa}{ \end{eqnarray} }
\newcommand{\beqano}{ \begin{eqnarray*} }
\newcommand{\eeqano}{ \end{eqnarray*} }

\renewcommand{\to}[0]{\longrightarrow}
\newcommand{\kp}[2]{\langle #1 , #2 \rangle}
\newcommand{\connectsum}[2]{#1 \# #2}
\newcommand{\avg}[2][]{
  \ifthenelse{\equal{#1}{}}{
    \ifthenelse{\equal{#2}{}}{
      \mathfrak{m}%
    }{%
      \overline{#2}%
    }%
  }{
    \ifthenelse{\equal{#2}{}}{%
      \mathfrak{m}_{#1}%
    }{%
      \mathfrak{m}_{#1}(#2)}}}
\newcommand{\pb}[2]{\left\{ #1,#2 \right\}}

\makeatletter
\def\th@defn{%
  \thm@notefont{\bfseries\slshape} 
  \normalfont 
}
\def\th@theorema{%
  \thm@notefont{\bfseries} 
  \itshape 
}
\makeatother

\newtheorem{question}{Question}
\newtheorem*{question*}{Question}

\swapnumbers
\theoremstyle{theorema}
\newtheorem{theorem}{Theorem}[section]
\newtheorem*{theorem*}{Theorem}
\newtheorem{lemma}{Lemma}[section]
\newtheorem*{lemma*}{Lemma}

\theoremstyle{plain}
\newtheorem{proposition}{Proposition}[section]
\newtheorem*{proposition*}{Proposition}
\theoremstyle{theorema}

\newtheorem*{corollary*}{Corollary}

\theoremstyle{remark}
\newtheorem{remark}{Remark}[section]
\newtheorem*{remark*}{Remark}

\theoremstyle{defn}
\newtheorem{definition}{Definition}[section]
\newtheorem*{definition*}{Definition}

\newtheorem*{example*}{Example}

\newtheorem*{axiom*}{Axiom}

\newcounter{tem}
\renewcommand{\thetem}{\roman{tem}}
\newcommand{\tem}{\addtocounter{tem}{1}(\thetem)\ }
\newenvironment{itemise}[1][6mm]{
\setcounter{tem}{0}

\def\lmi{\leftmargini}
\setlength{\leftmargini}{#1}
\begin{enumerate}}{
\end{enumerate}
\setlength{\leftmargini}{\lmi}}

\newcommand{\inclcomments}[0]{1}
\newcommand{\includecomments}[1]{%
  \ifthenelse{\equal{#1}{}}{%
    \inclcomments{}%
  }{%
    \renewcommand{\inclcomments}[0]{#1}}%
}

\newsavebox{\mcommentbox}
\newenvironment{mcomment}[1][{true}]{%
  \def\includethiscomment{#1}
  \begin{lrbox}{\mcommentbox}%
    \begin{minipage}{\marginparwidth}%
      \raggedright\footnotesize\color{red}}{%
    \end{minipage}%
  \end{lrbox}%
  \marginpar{%
    \ifthenelse{\equal{{\inclcomments}}{{false}}\or\equal{{\includethiscomment}}{{seen}}}{%
      }{\usebox{\mcommentbox}}
  }%
}

\definecolor{gray}{rgb}{0.6,.6,0.6}
\definecolor{green}{rgb}{0,1,0}


%% file: e2-page.pdf_tex
\setlength{\unitlength}{4144sp}%
\begingroup\makeatletter\ifx\SetFigFont\undefined%
\gdef\SetFigFont#1#2#3#4#5{%
  \reset@font\fontsize{#1}{#2pt}%
  \fontfamily{#3}\fontseries{#4}\fontshape{#5}%
  \selectfont}%
\fi\endgroup%
\begin{picture}(0,0)%
\includegraphics{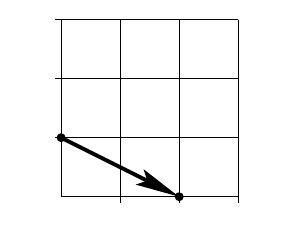}%
\end{picture}%
\begin{picture}(0,0)(886,-2621)
\put(1201,-1861){\makebox(0,0)[b]{\smash{{\SetFigFont{8}{9.6}{\familydefault}{\mddefault}{\updefault}{\color[rgb]{0,0,0}$1$}%
}}}}
\put(1201,-1411){\makebox(0,0)[b]{\smash{{\SetFigFont{8}{9.6}{\familydefault}{\mddefault}{\updefault}{\color[rgb]{0,0,0}$2$}%
}}}}
\put(1201,-961){\makebox(0,0)[b]{\smash{{\SetFigFont{8}{9.6}{\familydefault}{\mddefault}{\updefault}{\color[rgb]{0,0,0}$3$}%
}}}}
\put(1801,-2561){\makebox(0,0)[b]{\smash{{\SetFigFont{8}{9.6}{\familydefault}{\mddefault}{\updefault}{\color[rgb]{0,0,0}$1$}%
}}}}
\put(2296,-2561){\makebox(0,0)[b]{\smash{{\SetFigFont{8}{9.6}{\familydefault}{\mddefault}{\updefault}{\color[rgb]{0,0,0}$2$}%
}}}}
\put(2701,-2561){\makebox(0,0)[b]{\smash{{\SetFigFont{8}{9.6}{\familydefault}{\mddefault}{\updefault}{\color[rgb]{0,0,0}$3$}%
}}}}
\put(3001,-2361){\makebox(0,0)[lb]{\smash{{\SetFigFont{8}{9.6}{\familydefault}{\mddefault}{\updefault}{\color[rgb]{0,0,0}$H^*(B)$}%
}}}}
\put(901,-961){\makebox(0,0)[rb]{\smash{{\SetFigFont{8}{9.6}{\familydefault}{\mddefault}{\updefault}{\color[rgb]{0,0,0}$H^*(\T^d)$}%
}}}}
\end{picture}%
\begin{picture}(2130,1795)
\end{picture}%

%% file: e2-page.eepic
\setlength{\unitlength}{0.00087489in}
\begingroup\makeatletter\ifx\SetFigFont\undefined%
\gdef\SetFigFont#1#2#3#4#5{%
  \reset@font\fontsize{#1}{#2pt}%
  \fontfamily{#3}\fontseries{#4}\fontshape{#5}%
  \selectfont}%
\fi\endgroup%
{\newcommand{\dashlinestretch}{30}
\begin{picture}(2130,1810)(0,-10)
\put(465,760){\blacken\ellipse{60}{60}}
\put(465,760){\ellipse{60}{60}}
\put(1365,310){\blacken\ellipse{60}{60}}
\put(1365,310){\ellipse{60}{60}}
\path(420,760)(1815,760)
\path(420,1210)(1815,1210)
\path(420,1660)(1815,1660)
\path(1365,1660)(1365,265)
\path(1815,1660)(1815,265)
\thicklines
\path(465,760)(1365,310)
\blacken\path(1244.252,336.833)(1365.000,310.000)(1271.085,390.498)(1289.868,347.566)(1244.252,336.833)
\thinlines
\path(465,1660)(465,310)(1815,310)
\path(915,1660)(915,265)
\put(315,760){\makebox(0,0)[b]{\smash{{\SetFigFont{8}{9.6}{\familydefault}{\mddefault}{\updefault}$1$}}}}
\put(315,1210){\makebox(0,0)[b]{\smash{{\SetFigFont{8}{9.6}{\familydefault}{\mddefault}{\updefault}$2$}}}}
\put(315,1660){\makebox(0,0)[b]{\smash{{\SetFigFont{8}{9.6}{\familydefault}{\mddefault}{\updefault}$3$}}}}
\put(915,60){\makebox(0,0)[b]{\smash{{\SetFigFont{8}{9.6}{\familydefault}{\mddefault}{\updefault}$1$}}}}
\put(1410,60){\makebox(0,0)[b]{\smash{{\SetFigFont{8}{9.6}{\familydefault}{\mddefault}{\updefault}$2$}}}}
\put(1815,60){\makebox(0,0)[b]{\smash{{\SetFigFont{8}{9.6}{\familydefault}{\mddefault}{\updefault}$3$}}}}
\put(2115,260){\makebox(0,0)[lb]{\smash{{\SetFigFont{8}{9.6}{\familydefault}{\mddefault}{\updefault}$H^*(B)$}}}}
\put(15,1660){\makebox(0,0)[rb]{\smash{{\SetFigFont{8}{9.6}{\familydefault}{\mddefault}{\updefault}$H^*(\T^d)$}}}}
\end{picture}
}